\documentclass[12pt,reqno]{amsart}

\usepackage{amsbsy}
\usepackage{amscd}
\usepackage{amsfonts}
\usepackage{amsmath}
\usepackage{amssymb}
\usepackage{amsthm}
\usepackage{amsxtra}

\usepackage[utf8]{inputenc}
\usepackage[margin=1in]{geometry}

\usepackage{datetime}
\usepackage{tikz-cd}
\usepackage{tikz}
\usepackage{todonotes}
\usepackage{verbatim}

\theoremstyle{definition}

\newtheorem{thm}{Theorem}[section]
\newtheorem*{thm*}{Theorem}

\newtheorem*{conj*}{Conjecture}

\newtheorem*{conv*}{Convention}
\newtheorem{cor}[thm]{Corollary}
\newtheorem*{cor*}{Corollary}

\newtheorem*{defn*}{Definition}

\newtheorem*{exa*}{Example}
\newtheorem{exc}[thm]{Exercise}
\newtheorem*{exc*}{Exercise}

\newtheorem*{fact*}{Fact}
\newtheorem{lem}[thm]{Lemma}
\newtheorem*{lem*}{Lemma}

\newtheorem*{prob*}{Problem}

\newtheorem*{prop*}{Proposition}

\newtheorem*{ques*}{Question}
\newtheorem{rmk}[thm]{Remark}
\newtheorem*{rmk*}{Remark}

\newcommand{\mc}{\mathcal}

\newcommand{\C}{\mathbb{C}}

\newcommand{\Q}{\mathbb{Q}}

\newcommand{\R}{\mathbb{R}}
\newcommand{\Z}{\mathbb{Z}}

\newcommand{\cC}{\mathcal{C}}

\newcommand{\cM}{\mathcal{M}}

\newcommand{\cU}{\mathcal{U}}
\newcommand{\cV}{\mathcal{V}}

\newcommand{\al}{\alpha}
\newcommand{\bet}{\beta}

\newcommand{\gam}{\gamma}

\newcommand{\del}{\delta}
\newcommand{\eps}{\varepsilon}

\newcommand{\kap}{\kappa}
\newcommand{\Lam}{\Lambda}

\newcommand{\Om}{\Omega}
\newcommand{\om}{\omega}

\newcommand{\ra}{\rightarrow}

\newcommand{\ol}{\overline}
\newcommand{\pr}{\prime}

\newcommand{\wt}{\widetilde}

\newcommand{\lr}[1]{\langle #1 \rangle}

\newcommand{\sm}{\setminus}

\DeclareMathOperator{\GL}{GL}

\DeclareMathOperator{\Sp}{Sp}

\DeclareMathOperator{\Area}{Area}

\DeclareMathOperator{\im}{Im}

\DeclareMathOperator{\lcm}{lcm}

\DeclareMathOperator{\Per}{Per}

\DeclareMathOperator{\rank}{rank}

\DeclareMathOperator{\Span}{span}

\DeclareMathOperator*{\vol}{vol}

\newcommand{\be}{\begin{equation*}}
\newcommand{\ee}{\end{equation*}}
\newcommand{\bex}{\begin{exc}}
\newcommand{\eex}{\end{exc}}
\newcommand{\bpf}{\begin{proof}}
\newcommand{\epf}{\end{proof}}

\settimeformat{hhmmsstime}

\title{Isoperiodic forms and invariant subvarieties of connected strata}
\author[Karl Winsor]{Karl Winsor}
\date{May 31, 2026}

\begin{document}

\begin{abstract}
Strata of holomorphic $1$-forms have an absolute period foliation given by varying $1$-forms while keeping their integrals along closed loops fixed. In this paper, we classify the leaf closures of this foliation when the stratum is connected, outside of a subvariety of high codimension.
\end{abstract}

\maketitle


\section{Introduction}

Let $\Om\cM_g(\kap)$ be a stratum of holomorphic $1$-forms with at least $2$ distinct zeros. The stratum $\Om\cM_g(\kap)$ has an absolute period foliation obtained by varying holomorphic $1$-forms while keeping their integrals along closed loops fixed. The purpose of this paper is to contribute to the study of the dynamics of this foliation, and to make progress toward an analogue of Ratner's orbit closure theorem \cite{Rat:ICM} for the closures of absolute period leaves in strata.

Previously, Calsamiglia-Deroin-Francaviglia \cite{CDF:transfer} established such an analogue in the principal stratum $\Om\cM_g(1,\dots,1)$. They obtained their leaf closure theorem by showing that holomorphic $1$-forms in this stratum with the same absolute periods lie on the same absolute period leaf. Using this topological result, one can transfer dynamical results for the action of $\Sp(2g,\Z)$ on the space of homomorphisms $H_1(S_g;\Z) \ra \C$ due to Kapovich \cite{Kap:periods} to dynamical results for leaf closures of the absolute period foliation.

By the Riemann bilinear relations, a holomorphic $1$-form $(X,\om)$ has positive area in the natural metric determined by $\om$, and the group of absolute periods $\Per(\om)$ contains a lattice in $\C$. Let $\Lam(\om)$ denote the closure of $\Per(\om)$ in $\C$. Fix $a > 0$ and let $\Lam \subset \C$ be a closed additive subgroup that contains a lattice in $\C$. For $\cC$ a connected component of $\Om\cM_g(\kap)$, define $\cC_a$ to be the space of $1$-forms in $\cC$ with area $a$, and define
\be
\cC_a^\Lam = \{(X,\om) \in \cC_a : \Lam(\om) \subset \Lam \text{ and meets every component of } \Lam \} .
\ee
The main result of this paper is the following.

\begin{thm} \label{thm:leafclosures}
Fix $g \geq 3$. Let $\cC$ be either a connected stratum or the nonhyperelliptic component of $\Om\cM_g(g-1,g-1)$ with $g - 1$ odd. There is an algebraic subvariety $\cV \subsetneq \cC$ of positive codimension with the following property. For any $(X,\om) \in \cC \sm \cV$ with area $a > 0$, the closure in $\cC$ of the absolute period leaf through $(X,\om)$ is a connected component of $\cC_a^{\Lam(\om)}$.
\end{thm}

The remaining stratum components can also be addressed using the same general approach in this paper. However, there are several additional technical challenges for these stratum components, so for reasons of length this work is deferred to a separate paper.

\begin{rmk} \label{rmk:highcodim}
The subvariety $\mc{V}$ is invariant under both the absolute period foliation of $\mc{C}$ and the $\GL^+(2,\R)$-action on $\mc{C}$. Following Apisa-Wright \cite{AW:highrank}, $\GL^+(2,\R)$-orbit closures in are called {\em invariant subvarieties}. Theorem 1.1 in \cite{AW:highrank} implies that when $\cC$ is a connected stratum, the complex codimension of the invariant subvariety $\cV$ in $\cC$ must be at least $g - 1$.
\end{rmk}

\paragraph{\bf Methods.} Our approach to Theorem \ref{thm:leafclosures} is based on a new elementary proof of Kapovich's orbit closure theorem for the $\Sp(2g,\Z)$-action on homomorphisms $H_1(S_g;\Z) \ra \C$. Our proof focuses on the dynamics of abelian subgroups of $\Sp(2g,\Z)$ generated by shears and factor mixes in a given symplectic basis. This proof can be adapted to study the dynamics of absolute period leaves by, roughly speaking, showing that certain nearby $1$-forms whose absolute periods are related by an element of an abelian subgroup of Dehn twists lie on the same absolute period leaf. At the end of our proof, we rely on results in \cite{Win:ergodic} that show the existence of dense absolute period leaves in the loci $\cC_a^{\cC}$ and $\cC_a^{\R + i\Z}$, and on the theorems of Eskin-Mirzakhani-Mohammadi \cite{EMM:closures} and \cite{Fil:splitting}. \\

\paragraph{\bf Related work.} The dynamics of the absolute period foliation and its subfoliations have been studied in many other papers now, including \cite{CDF:transfer}, \cite{CW:realrel}, \cite{Ham:ergodicity}, \cite{HW:Rel}, \cite{McM:isoperiodic}, \cite{Osp:realrel}, \cite{Win:ergodic}, \cite{Win:dense}, \cite{Win:saturated}, \cite{Ygo:dense}.

The current paper is most directly inspired by the work in \cite{CDF:transfer} and \cite{Win:saturated}. In \cite{CDF:transfer}, a transfer principle is proven allowing one to deduce dynamical results for the absolute period foliation of the principal stratum from dynamical results from \cite{Kap:periods} for the $\Sp(2g,\Z)$-action on the space of positive-volume homomorphisms $H_1(S_g;\Z) \ra \C$. In particular, \cite{CDF:transfer} gives a classification of leaf closures in the principal stratum and shows that the absolute period foliation is ergodic on these closed sets. The approach in \cite{CDF:transfer} is inductive and builds off of an observation in \cite{McM:isoperiodic} in genus $2$ and $3$. In \cite{Win:saturated}, a shorter proof of the leaf closure classification theorem of \cite{CDF:transfer} is obtained by applying the rigidity theorems of \cite{EMM:closures} for $\GL^+(2,\R)$-orbit closures in strata. The current paper provides a step toward extending the methods of \cite{CDF:transfer} and \cite{Win:saturated} to general strata of holomorphic $1$-forms in the hopes of completely classifying absolute period leaf closures in strata.

In \cite{Ham:ergodicity}, a simpler proof of the ergodicity of the absolute period foliation of the principal stratum is given, also using induction on genus and using \cite{McM:isoperiodic} for base cases. Examples of dense absolute period leaves in certain strata are given in \cite{HW:Rel} and \cite{Ygo:dense}. In \cite{Win:dense}, the existence of dense real Rel flow orbits in all stratum components with multiple zeros is proven. In \cite{CW:realrel}, it is shown that real Rel flows are ergodic on all stratum components with multiple zeros. In \cite{Osp:realrel}, the existence of real Rel flow orbits in $\Om\cM_2(1,1)$ that are not divergent and non-recurrent is proven.

Lastly, the dynamics and intrinsic structure of absolute period leaves are also an interesting topic in the setting of strata of meromorphic differentials. We refer to \cite{CD:meromorphic} and \cite{FTZ:meromorphic} for recent examples of work on this topic. \\

\paragraph{\bf Acknowledgements.} During the preparation of this paper, the author was supported by an NSF MSPRF under grant number DMS-2303185.


\section{Revisiting Kapovich's Orbit Closure Theorem} \label{sec:Kapovich}

Let $S_g$ be the connected closed oriented surface of genus $g$. The cohomology group $H^1(S_g;\C)$ is identified with the space of homomorphisms $H_1(S_g;\Z) \ra \C$. We will denote $H_1 = H_1(S_g;\Z)$. The {\em volume} of $p \in H^1$ is defined by
\be
\vol(p) = \sum_{j=1}^g \im(\ol{p(a_j)} p(b_j))
\ee
where $\{a_j,b_j\}_{j=1}^g$ is any symplectic basis for $H_1$. We are mainly interested in the space of positive-volume homomorphisms
\be
V = \{p \in H^1 : \vol(p) > 0\} .
\ee
For $p \in V$, let $\Lam(p)$ be the closure of the image of $p$ in $\C$. When $\vol(p) \neq 0$, $\Lam(p)$ is a closed subgroup of $\C$ containing a lattice, so up to the action of $\GL^+(2,\R)$ there are $3$ possibilities for $\Lam(p)$: $\Z + i\Z$, $\R + i\Z$, and $\C$.

Let $\Sp(H_1)$ denote the group of automorphisms of $H_1$ that preserve the algebraic intersection pairing. The group $\Sp(H_1)$ acts on $V$ by precomposition: for $p \in V$ and $M \in \Sp(H_1)$, we have $M \cdot p = p \circ M^{-1}$. The action of $\Sp(H_1)$ on $V$ preserves both $\vol(p)$ and $\Lam(p)$. Thus, for $a > 0$ and $\Lam \subset \C$ a closed subgroup, we define
\begin{align*}
V_a &= \{p \in V : \vol(p) = a\} \\
V_a^\Lam &= \{p \in V_a : \Lam(p) \subset \Lam \text{ and meets every component of } \Lam \} .
\end{align*}
When $\Lam = \C$, we have $V_a^\Lam = V_a$. When $\Lam \cong \R + i\Z$, a homomorphism $p \in V_a$ is in $V_a^\Lam$ if and only if $p(H_1) \subset \Lam$ and $p(H_1) \not\subset n\Lam$ for all integers $n \geq 2$. When $\Lam \cong \Z + i\Z$, a homomorphism $p \in V_a$ is in $V_a^\Lam$ if and only if $p(H_1) = \Lam$.

The question of which homomorphisms in $V$ arise from holomorphic $1$-forms in genus $g$ was originally answered by Haupt \cite{Hau:periods}. Kapovich rediscovered Haupt's theorem in \cite{Kap:periods} and provided a very different proof based on classifying the $\Sp(H_1)$-orbit closures in $V$, which uses Ratner's orbit closure theorem and other deep theorems from homogeneous dynamics. In this section, we give a new elementary proof of (most of) Kapovich's $\Sp(H_1)$-orbit closure classification, by analyzing the actions of abelian subgroups of $\Sp(H_1)$ on $V$.

\begin{thm} \label{thm:Kapovich} (\cite{Kap:periods}, see Proposition 3.10 in \cite{CDF:transfer}) Suppose $g \geq 3$, and fix $p \in V$. Let $a = \vol(p)$ and $\Lam = \Lam(p)$. The closure of $\Sp(H_1) \cdot p$ in $V$ is given by $V_a^\Lam$.
\end{thm}

We do not provide a new argument in the case where $\Lam(p) \cong \Z + i\Z$, so we will not address this case here. We only remark that it is easy to see that the orbit $\Sp(H_1) \cdot p$ is closed in this case. \\

\paragraph{\bf $U_A$-orbit closures.} To set up our argument, fix a symplectic basis $\{a_j,b_j\}_{j=1}^g$ for $H_1$. Denote the $\Z$-span of $a_1,\dots,a_g$ by $A$, and the $\Z$-span of $b_1,\dots,b_g$ by $B$. Then $A$ is a Lagrangian in $H_1$, and its pointwise stabilizer in $\Sp(H_1)$ is a free abelian subgroup $U_A$ of rank $g(g+1)/2$ generated by the following two types of transformations:
\begin{itemize}
    \item Shears $s_j$, $1 \leq j \leq g$, given by $s_j(b_j) = b_j - a_j$.
    \item Factor mixes $m_{jk}$, $1 \leq j < k \leq g$, given by $m_{jk}(b_j) = b_j - a_k$, $m_{jk}(b_k) = b_k - a_j$.
\end{itemize}
Here, any basis elements not mentioned are fixed.

Fix $p \in V$, let $z_j = p(a_j)$ for $1 \leq j \leq g$, and let $z = (z_1,\dots,z_g)$. Consider the {\em affine slice}
\be
W(z) = \left\{(w_1,\dots,w_g) \in \C^g : \sum_{j=1}^g \im(\ol{z_j}w_j) = \vol(p) \right\} .
\ee
Note that $W(z)$ is a real affine subspace of $\C^g$ of real dimension $2g - 1$. The action of $U_A$ on $V$ only changes the $B$-periods of a homomorphism while leaving the $A$-periods fixed. Thus, by forgetting the $A$-periods, the action of $U_A$ on $V$ induces an action on each affine slice $W(z)$ with $z \neq 0$. Moreover, this action of $U_A$ on $W(z)$ is by translations in $\C^g$. To see this, let $e_1,\dots,e_g$ be the standard basis of $\C^g$. Then for $w = (w_1,\dots,w_g) \in \C^g$, we have
\be
s_j \cdot w = w + z_j e_j, \quad m_{jk} \cdot w = w + z_k e_j + z_j e_k .
\ee

We begin by studying $U_A$-orbit closures in $W(z) \subset \C^g$ in the case where $g = 3$. We say that two complex numbers $u_1,u_2$ are {\em nonparallel} if they are nonzero and $u_2/u_1 \notin \R$. Otherwise, we say that $u_1,u_2$ are {\em parallel}. Fix $z = (z_1,z_2,z_3) \in \C^3$, and suppose that the components $z_1,z_2,z_3$ are pairwise nonparallel. We can then write $z_3 = xz_1 + yz_2$ for unique $x,y \in \R$. Denote
\be
d(z) = -1 + \dim_\Q \Span_\Q\lr{1,x,y,x^2,xy,y^2} .
\ee
It turns out that $d(z)$ is well-defined under permutations of $z_1,z_2,z_3$.

\begin{lem} \label{lem:UAdim}
Fix $z = (z_1,z_2,z_3) \in \C^3$ such that $z_1,z_2,z_3$ are pairwise nonparallel. The $U_A$-orbit closures in $W(z)$ are translates of closed subgroups of real dimension $d(z)$.
\end{lem}

\begin{proof}
Recall that $W(z)$ is a real affine subspace of $\C^3$ of real dimension $5$. Write $W(z) = w_0 + V(z)$, where $w_0 \in \C^3$ and $V(z)$ is the real vector space
\be
V(z) = \{(w_1,w_2,w_3) \in \C^3 : \sum_{j=1}^3 \im(\ol{z}_j w_j) = 0\} .
\ee
Since $g = 3$, the group $U_A$ is generated by $3$ shears $s_j$ and $3$ factor mixes $m_{jk}$, and the associated translation vectors in $\C^3$ are $z_j e_j$, $1 \leq j \leq 3$, and $z_j e_k + z_k e_j$, $1 \leq j < k \leq 3$. These translation vectors lie in $V(z)$. Relabel these vectors by
\be
v_1 = z_1 e_1, v_2 = z_2 e_2, v_3 = z_3 e_3, v_4 = z_2 e_1 + z_1 e_2, v_5 = z_3 e_2 + z_2 e_3, v_6 = z_3 e_1 + z_1 e_3 .
\ee

The first $5$ translation vectors $v_1,\dots,v_5$ generate a lattice $\Lam(z)$ in $V(z)$. Equivalently, these vectors are $\R$-linearly independent. Indeed, an $\R$-linear relation $\sum_{j=1}^5 t_j v_j = 0$ would imply
\be
t_1 z_1 + t_4 z_2 = 0, \quad t_2 z_2 + t_4 z_1 + t_5 z_3 = 0, \quad t_3 z_3 + t_5 z_2 = 0 .
\ee
Since $z_1,z_2$ are nonparallel, $t_1 = 0$ and $t_4 = 0$. Since $z_2,z_3$ are nonparallel, $t_3 = 0$ and $t_5 = 0$. Then $t_2 z_2 = 0$ implies $t_2 = 0$, and thus the $\R$-linear relation is trivial.

Write $z_3 = x z_1 + y z_2$ with $x,y \in \R$. Note that $x \neq 0$ and $y \neq 0$ since $z_1,z_2,z_3$ are nonparallel. The vector $v_6$ is an $\R$-linear combination of $v_1,\dots,v_5$ in a unique way, given by
\be
v_6 = x v_1 + \frac{y^2}{x} v_2 + \frac{1}{x} v_3 + y v_4 - \frac{y}{x} v_5 .
\ee
Let $u = (x, y^2/x, 1/x, y, -y/x)$, and consider the $\R$-linear isomorphism $L : \R^5 \ra V(z)$, $L(t_1,\dots,t_5) = \sum_{j=1}^5 t_j v_j$. The preimage under $L$ of the $U_A$-orbit of $0 \in V(z)$ is given by $\Z^5 + \Z u \subset \R^5$. Let $O(u) \subset \R^5$ be the closure of $\Z^5 + \Z u$.

By Kronecker's theorem, the closure of $\Z u$ in $\R^5 / \Z^5$ is the smallest closed subgroup of $\R^5 / \Z^5$ containing $u$. Equivalently in $\R^5$, letting $R(u)$ be the set of $v \in \Z^5$ such that $v \cdot u \in \Z$, we have
\be
O(u) = \{r \in \R^5 : r \cdot v \in \Z \text{ for all } v \in R(u) \} .
\ee
The identity component of $O(u)$ is then the real vector subspace given by
\be
O_0(u) = \{r \in \R^5 : r \cdot v = 0 \text{ for all } v \in R(u) \} ,
\ee
and $O(u)$ is a union of translates of $O_0(u)$. Since $R(u)$ is identified with the group of $\Z$-linear relations among $1$ and the components of $u$, we then have
\be
\dim_\R O_0(u) = 5 - \rank R(u) = -1 + \dim_\Q \Span_\Q \{1,x,y^2/x,1/x,y,-y/x\} .
\ee
Scaling $1$ and the components of $u$ all by $x$, we conclude that this dimension is $d(z)$.
\end{proof}

Next, we spell out the possibilities for $d(z)$ more explicitly.

\begin{lem} \label{lem:dzcases}
Fix $x,y \in \R$, and let $d = -1 + \dim_\Q \Span_\Q\lr{1,x,y,x^2,xy,y^2}$.
\begin{itemize}
    \item $d = 0$ if and only if $x,y \in \Q$.
    \item $d = 1$ if and only if $\Q(x,y)$ is a real quadratic field.
    \item $d = 2$ if and only if $\Q(x,y)$ is a real cubic field, or $1,x,y$ are $\Q$-linearly dependent and $\Q(x,y)$ has degree at least $3$.
    \item If $d \geq 3$, then $1,x,y$ are $\Q$-linearly independent.
\end{itemize}
\end{lem}

\begin{proof}
The $d = 0$ case is clear.

If $d = 1$, then after swapping $x$ and $y$, we may assume $x$ is irrational. Since $1,x,x^2$ are $\Q$-linearly dependent, $\Q(x) = \Q + \Q x$ is a real quadratic field. Since $1,x,y$ are $\Q$-linearly dependent, $\Q(x,y) = \Q(x)$. Conversely, if $\Q(x,y)$ is a real quadratic field, then at least one of $x$ and $y$ is irrational, so since $1,x,y,x^2,xy,y^2$ all lie in $\Q(x,y)$, their $\Q$-span has $\Q$-dimension $2$ and so $d = 1$.

Suppose $1,x,y$ are $\Q$-linearly dependent. By swapping $x$ and $y$, we may assume $y \in \Span_\Q\{1,x\}$. Then $x^2,xy,y^2$ all lie in $\Span_\Q\{1,x,x^2\}$, so $d \leq 2$ with equality if and only if $1,x,x^2$ are $\Q$-linearly independent. In the $d = 2$ case, the field $\Q(x,y) = \Q(x)$ has degree at least $3$ (possibly infinite degree). Note that when $d \geq 3$, we have shown $1,x,y$ are $\Q$-linearly independent.

Lastly, suppose $1,x,y$ are $\Q$-linearly independent. Then $d \geq 2$, with equality if and only if $F = \Span_\Q\{1,x,y\}$ is a $\Q$-algebra. Since $F \subset \R$, this algebra is finite-dimensional with no zero divisors, and is thus a field. This means $\Q(x,y) = F$ is a cubic field.
\end{proof}

\paragraph{\bf Bases with nonparallel values.} Suppose $p \in V$ is such that $\Lam(p) = \C$ or $\Lam(p) \cong \R + i\Z$. In order to show that the $\Sp(H_1)$-orbit closure of $p$ is large, we will find a triple of elements of $H_1$ that span an isotropic subspace $A$ such that $p(A)$ is a nondiscrete subgroup of $\C$, and then we will apply Lemma \ref{lem:UAdim}. We will then bootstrap to find other homomorphisms in the $\Sp(H_1)$-orbit closure and other triples with stronger incommensurability properties, meaning higher values of $d(z)$, to conclude. To begin this bootstrap argmuent, we will need a symplectic basis of $H_1$ on which $p$ takes nonparallel values.

\begin{lem} \label{lem:nonparallelbasis}
For any $p \in V$, there is a symplectic basis $c_1,\dots,c_{2g}$ of $H_1$ such that the values $p(c_1),\dots,p(c_{2g})$ are pairwise nonparallel.
\end{lem}

\begin{proof}
Fix an arbitrary symplectic basis for $H_1$ and denote its elements by $c_1,\dots,c_{2g}$. Since $p \neq 0$, $\ker(p)$ and $c_j^\perp$ are submodules of $H_1$ of rank less than $2g$. Thus, there exists $e \in H_1$ such that $p(e) \neq 0$ and $e \cdot c_j \neq 0$ for all $j$. Let $T \in \Sp(H_1)$ be the transvection given by $T(c_j) = c_j + (c_j \cdot e) e$ for all $j$. For each $j$, since $p(e) \neq 0$ and $c_j \cdot e \neq 0$, there is at most one $n_j \in \Z$ such that $p(T^{n_j}(c_j)) = p(c_j) + n_j(c_j \cdot e)p(e) = 0$. In particular, there is $N \in \Z$ such that the symplectic basis $T^N(c_1),\dots,T^N(c_{2g})$ satisfies $p(T^N(c_j)) \neq 0$ for all $j$.

Thus, we may now assume that $p(c_j) \neq 0$ for all $j$. Since $\vol(p) \neq 0$, the image of $p$ is not contained in a real line through the origin, so $M_j = \{e \in H_1 : p(e) \in \R p(c_j)\}$ is a submodule of $H_1$ of rank less than $2g$. For any $j < k$ such that $p(c_j) / p(c_k) \in \Q$, let $q_{jk} = p(c_j) / p(c_k)$. Since $c_j,c_k$ are distinct basis elements, $M_{jk} = \{e \in H_1 : e \cdot c_j = q_{jk} (e \cdot c_k)\}$ is a submodule of $H_1$ of rank less than $2g$. Then as before, there is $e \in H_1$ such that $e \cdot c_j \neq 0$ and $p(e) \neq \R p(c_j)$ for all $j$, and such that $e \cdot c_j \neq q_{jk} (e \cdot c_k)$ for all $j < k$ with $p(c_j) / p(c_k) \in \Q$. Let $T$ be the transvection given by $T(c_j) = c_j + (c_j \cdot e) e$ for all $j$. Fix $n \in \Z$. Since $p(e) \notin p(c_j)$, we have $p(T^n(c_j)) = p(c_j) + n (c_j \cdot e) p(e) \neq 0$. And if $j < k$ and $p(T^n(c_j))$ and $p(T^n(c_k))$ are parallel, then we have
\begin{align*}
0 &= \im((\ol{p(c_j) + n(c_j \cdot e)p(e)})(p(c_k) + n(c_k \cdot e)p(e))) = 0 \\
& = \im(\ol{p(c_j)}p(c_k)) + n\left((c_k \cdot e)\im(\ol{p(c_j)}p(e)) - (c_j \cdot e)\im(\ol{p(c_k)}p(e))\right)
\end{align*}
which is an affine equation in $n$ and has at most one solution in $n$ as long as at least one coefficient is nonzero. If $p(c_j)$ and $p(c_k)$ are nonparallel, then the first coefficient is $\im(\ol{p(c_j)}p(c_k)) \neq 0$, so suppose $p(c_j)$ and $p(c_k)$ are parallel. If $p(c_j) = r p(c_k)$ with $r \in \R$, then the second coefficient becomes
\be
(r(c_k \cdot e) - (c_j \cdot e))\im(\ol{p(c_k)}p(e)) .
\ee
Since $p(e) \notin \R p(c_k)$, we have $\im(\ol{p(c_k)}p(e)) \neq 0$. And since $c_k \cdot e$ and $c_j \cdot e$ are integers, either $r \notin \Q$, or $r \in \Q$ and $e \notin M_{jk}$, so $r(c_k \cdot e) - (c_j \cdot e) \neq 0$. Thus, for all but finitely many $n \in \Z$, the symplectic basis $T^n(c_1),\dots,T^n(c_{2g})$ has the desired property.
\end{proof}

Two lattices $\Lam_1,\Lam_2 \subset \C$ are {\em commensurable} if $\Lam_1 \cap \Lam_2$ has finite index in both $\Lam_1$ and $\Lam_2$. In this case, $\Lam_1 + \Lam_2$ is again a lattice in $\C$.

\begin{lem} \label{lem:nondiscreteLag3}
Suppose $g \geq 3$. Fix $p \in V$ and suppose $\Lam(p)$ is not discrete. Then there are $a_1,a_2,a_3$ that are part of a Lagrangian in $H_1$ such that $p(a_1),p(a_2),p(a_3)$ generate a nondiscrete subgroup of $\C$.
\end{lem}

\begin{proof}
By Lemma \ref{lem:nonparallelbasis}, there is a symplectic basis $\{a_j,b_j\}_{j=1}^g$ on which $p$ takes nonparallel values. For $1 \leq j \leq g$, let $\Lam_j$ be the lattice in $\C$ generated by $p(a_j),p(b_j)$. The lattices $\Lam_j$ cannot be pairwise commensurable, otherwise $p(H_1) = \Lam_1 + \cdots + \Lam_g$ would be discrete, contradicting our assumption that $\Lam(p) = \ol{p(H_1)}$ is not discrete. By reordering, we may assume $\Lam_1$ and $\Lam_2$ are incommensurable.

Let $\Lam_0^\pr$ be the group generated by $p(a_1),p(a_2),p(a_3)$, and for $1 \leq j \leq 3$, let $\Lam_j^\pr$ be the group generated by $p(b_j),p(a_{j+1}),p(a_{j+2})$, indices taken modulo $3$. Note that for $0 \leq j \leq 3$, the given generators of $\Lam_j^\pr$ are a triple of isotropic elements of $H_1$. Since $p$ takes nonparallel values on the symplectic basis $\{a_j,b_j\}_{j=1}^g$, the intersection $\Lam_0^\pr \cap \Lam_j^\pr$ contains a lattice for $1 \leq j \leq 3$. Then if $\Lam_0^\pr,\Lam_1^\pr,\Lam_2^\pr,\Lam_3^\pr$ were all lattices, they would be pairwise commensurable. However, then $\Lam_1 + \Lam_2 \subset \Lam_0^\pr + \Lam_1^\pr + \Lam_2^\pr + \Lam_3^\pr$ would also be a lattice, contradicting that $\Lam_1$ and $\Lam_2$ are incommensurable. Thus, at least one of the $\Lam_j^\pr$, $0 \leq j \leq 3$, must be nondiscrete.
\end{proof}

\paragraph{\bf The bootstrap argument.} Throughout this subsection, we assume $g \geq 3$. Fix $p \in V$, and fix a symplectic basis $\{a_j,b_j\}_{j=1}^g$ of $H_1$ such that the values $p(a_j),p(b_j)$, $1 \leq j \leq g$, are pairwise nonparallel, using Lemma \ref{lem:nonparallelbasis}. Let $c_1,c_2,c_3$ be distinct elements of this basis that span an isotropic subspace, and let $A = \Span_\Z \{c_1,c_2,c_3\}$. The {\em $p$-dimension of $A$} is defined to be $d(p(c_1),p(c_2),p(c_3))$. If $p(c_3) = x p(c_1) + y p(c_2)$ with $1,x,y \in \R$ being $\Q$-linearly independent, we say that $A$ has {\em $p$-type 1}. If $1,x,y$ are $\Q$-linearly dependent, we say that $A$ has {\em $p$-type 2}. We extend our notation to define $U_A$ as being generated by the shears and factor mixes coming from $A$.

\begin{lem} \label{lem:UAdim1}
Suppose that $A = \Span_\Z\{c_1,c_2,c_3\}$ has $p$-dimension $1$. Then there is $p^\pr$ in the $U_A$-orbit closure of $p$ and $A^\pr = \Span_\Z\{c_1^\pr,c_2^\pr,c_3^\pr\}$ of $p^\pr$-dimension at least $2$, and of $p^\pr$-type $2$ in the case of equality.
\end{lem}

\begin{proof}
By permuting and negating the elements of our symplectic basis $\{a_j,b_j\}_{j=1}^g$ of $H_1$ if necessary, we may assume that $c_j = a_j$ for $1 \leq j \leq 3$. Let $z = (z_1,z_2,z_3) = (p(a_1),p(a_2),p(a_3))$. Since $d(z) = 1$, by Lemma \ref{lem:UAdim} the real dimension of the $U_A$-orbit closure of $w = (p(b_1),p(b_2),p(b_3))$ in $W(z)$ is $1$, so this $U_A$-orbit closure contains $w + \R u$ for some nonzero $u \in \C^3$. After reordering the $3$ coordinates, we may assume that the third coordinate is not constant along this $U_A$-orbit closure. In terms of $p$, this means that after reordering the pairs $(a_j,b_j)$, $1 \leq j \leq 3$, we may assume that $p(b_3)$ is not constant along the $U_A$-orbit closure of $p$.

Write $z_3 = x z_1 + y z_2$. By Lemma \ref{lem:dzcases}, $\Q(x,y)$ is a real quadratic field. In particular, for fixed $z_1$ and $z_2$, there are only countably many possibilities for $z_3$. Since $p(b_3)$ takes on uncountably many values along the $U_A$-orbit closure of $p$, there is $p^\pr$ in this $U_A$-orbit closure such that $p^\pr(b_3) = x^\pr z_1 + y^\pr z_2$ where $\Q(x^\pr,y^\pr)$ is not a field of degree at most $3$. Letting $A^\pr = \Span_\Z\{a_1,a_2,b_3\}$, by Lemma \ref{lem:dzcases} again, this means that $A^\pr$ has $p^\pr$-dimension at least $2$, with type $2$ in the case of equality.
\end{proof}

\begin{lem} \label{lem:UAtype1}
Suppose that $A = \Span_\Z\{c_1,c_2,c_3\}$ is of type $1$. There is $p^\pr$ in the $U_A$-orbit closure of $p$ and $A^\pr = \Span_\Z\{c_1^\pr,c_2^\pr,c_3^\pr\}$ of $p^\pr$-dimension $5$.
\end{lem}

\begin{proof}
Again, we may assume $c_j = a_j$ for $1 \leq j \leq 3$. Since $A$ is of type $1$, the group generated by $p(a_1),p(a_2),p(a_3)$ is dense in $\C$. Let $z = (z_1,z_2,z_3) = (p(a_1),p(a_2),p(a_3))$, and let $\pi_j : V \ra \C$ be the projection given by $\pi_j(p_0) = p_0(b_j)$. Let $O_A(p)$ be the closure of $U_A \cdot p$. We will show that $\pi_j(O_A(p)) = \C$. By translating, it is enough to show that $\pi_j(O_A) = \C$, where $O_A$ is the closure of $U_A \cdot 0$. By symmetry, permuting the $3$ coordinates, it is furthermore enough to show this for $\pi_1$.

Recall from the proof of Lemma \ref{lem:UAdim} our labelling $v_1,\dots,v_6$ of the translation vectors arising from $U_A$. Recall that the vectors $v_1,\dots,v_5$ form an $\R$-basis for $V(z)$, that $v_6 = x v_1 + \frac{y^2}{x} v_2 + \frac{1}{x} v_3 + y v_4 - \frac{y}{x} v_5$, and that $L^{-1}(O_A)$ is the closure of $\Z^5 + \Z u$ in $\R^5$ where $u = (x, y^2/x, 1/x, y, -y/x)$. By assumption, $1,x,y$ are $\Q$-linearly independent since $A$ is of type $1$, so the projection of $L^{-1}(O_A)$ onto the $1$-st and $4$-th coordinates is all of $\R^2$. We have
\be
\pi_1(L(t_1,\dots,t_5)) = t_1 z_1 + t_4 z_2 ,
\ee
and since $z_1$ and $z_2$ are nonparallel, their $\R$-span is all of $\C$. Since every pair $(t_1,t_4) \in \R^2$ arises from a vector in $L^{-1}(O_A)$, we have $\pi_1(O_A) = \C$.

We now have that for any $t_1,t_2 \in \R$, letting $w = t_1 z_1 + t_2 z_2 \in \C$, there is $p^\pr \in O_A(p)$ such that $p^\pr(b_3) = w$. Moreover, $p^\pr(a_1) = z_1$ and $p^\pr(a_2) = z_2$ since the $U_A$-action fixes the $A$-periods. Let $A^\pr = \Span_\Z \{a_1,a_2,b_3\}$. For generic $t_1,t_2 \in \R$, we have that $1,t_1,t_2,t_1^2,t_1 t_2,t_2^2$ are $\Q$-linearly independent, since each of the countably many possible $\Q$-linear relations only holds on the zero set of a nonzero quadratic polynomial in $t_1,t_2$. Then by definition, this means $A^\pr$ has $p^\pr$-dimension $5$.
\end{proof}

\begin{lem} \label{lem:UAdim2type2}
Suppose that $A = \Span_\Z\{c_1,c_2,c_3\}$ has $p$-dimension $2$ and is of $p$-type $2$. If $\Lam(p) = \C$, then there is $p^\pr$ in the $U_A$-orbit closure of $p$ and $A^\pr = \Span_\Z\{c_1^\pr,c_2^\pr,c_3^\pr\}$ such that $A^\pr$ has $p^\pr$-dimension at least $3$ and $p^\pr$-type $1$.
\end{lem}

\begin{proof}
Again, we may assume $c_j = a_j$ for $1 \leq j \leq 3$. Let $z_j = p(a_j)$ for $1 \leq j \leq 3$, $z = (z_1,z_2,z_3)$, and write $z_3 = x z_1 + y z_2$ with $x,y \in \R$. Since $A$ has $p$-dimension $2$ and is of type $2$, Lemma \ref{lem:dzcases} implies that $1,x,y$ are $\Q$-linearly dependent and that at least one of $x,y$ is irrational. Then there is a $\Q$-linear relation
\be
q_0 + q_1 x + q_2 y = 0
\ee
and this relation is unique up to scaling by $\Q^\ast$. With notation as in the proof of Lemma \ref{lem:UAtype1}, the argument from the proof of Lemma \ref{lem:UAtype1} gives us that each $\pi_j(O_A(p))$ contains the real line $\ell \subset \C$ given by the identity component of the closure of $\Z z_1 + \Z z_2 + \Z z_3$ in $\C$.

Since $A$ is of $p$-type $2$, we have $\ol{p(A)} = \ell + \Z z_0$ for some $z_0 \in \C \sm \ell$. Since $\Lam(p) = \C$, there is $h \in H_1$ such that $p(h) \notin \ell + \Q z_0$. If $p(b_j) \notin \ell + \Q z_0$ for some $1 \leq j \leq 3$, then after permuting the pairs $(a_j,b_j)$, we may assume that $p(b_3) \notin \ell + \Q z_0$. If not, then after subtracting an appropriate element of $\Z b_1 + \Z b_2 + \Z b_3$ from $h$ to get $h_0 \in A^\perp$, and replacing $b_3$ with $b_3 + h_0$, we again get $p(b_3) \notin \ell + \Q z_0$.

The set of $x z_1 + y z_2$ such that $1,x,y$ are $\Q$-linearly dependent forms a countably union of affine lines in the plane, and these lines are all distinct from $p(b_3) + \ell$ since $p(b_3) \notin \ell + \Q z_0$. Thus, each affine lines meets $p(b_3) + \ell$ in at most a single point. Thus, we can choose $p^\pr \in O_A$ such that $p^\pr(b_3) \in p(b_3) + \ell$ is not in any of these countably many affine lines, and then $A^\pr = \Span_\Z \{a_1,a_2,b_3\}$ is of $p^\pr$-type $1$.
\end{proof}

We are now ready to conclude our bootstrap argument.

\begin{proof} (of Theorem \ref{thm:Kapovich} when $\Lam(p)$ is not discrete.)

Fix $p \in V$ such that $\Lam(p)$ is not discrete. By scaling, we may assume $\vol(p) = 1$. By Lemma \ref{lem:nonparallelbasis}, there is a symplectic basis $\{a_j,b_j\}_{j=1}^g$ on which $p$ takes pairwise nonparallel values. By Lemma \ref{lem:nondiscreteLag3}, after permuting and negating the elements of this symplectic basis if necessary, we may assume that $p(a_1),p(a_2),p(a_3)$ generate a nondiscrete subgroup of $\C$. By Lemma \ref{lem:UAdim1}, after replacing $p$ with an element of its $U_A$-orbit closure, where $A = \Span_\Z\{a_1,a_2,a_3\}$, and permuting our symplectic basis again, we may assume that $d(p(a_1),p(a_2),p(a_3)) \geq 2$ and that $A$ is of $p$-type $2$ if equality holds. There are now two cases to consider. \\

\paragraph{\bf Case 1.} $\Lam(p) = \C$. If $d(p(a_1),p(a_2),p(a_3)) = 2$, then $A$ is of $p$-type $2$, and by Lemma \ref{lem:UAdim2type2}, after replacing $p$ with an element of its $U_A$-orbit closure and permuting our symplectic basis again, we may assume $A$ and $p$-type $1$. Finally, by Lemma \ref{lem:UAtype1}, after replacing $p$ with an element of its $U_A$-orbit closure and permuting our symplectic basis again, we may assume that $A$ has the maximum possible $p$-dimension of $5$.

We can now perturb $p(b_1),p(b_2),p(b_3)$ arbitrarily while preserving $\sum_{j=1}^3 \im(\ol{p(a_j)}p(b_j))$ and remain in the $U_A$-orbit closure of $p$. In particular, we can arrange so that $B_j = \Span_\Z \{b_1,b_2,b_j\}$ has $p$-dimension $5$ for all $3 \leq j \leq g$. Then by iterating this process to modify the $A$-periods of $p$, we get that the $\Sp(H_1)$-orbit closure of $p$ contains the entire affine slice $W(w)$ where $w = (p(b_1),\dots,p(b_g))$. Iterating again to modify the $B$-periods of $p$ similarly, and then one more time to modify the $A$-periods of $p$ again, we conclude that the $\Sp(H_1)$-orbit of $p$ is dense in $V_1$. \\

\paragraph{\bf Case 2.} $\Lam(p) \cong \R + i\Z$. We may assume $\Lam(p) = \R + i\Z$. Since $A$ must be of $p$-type $2$, we must have $d(p(a_1),p(a_2),p(a_3)) = 2$. In this case, we can perturb the real parts of $p(b_1),p(b_2),p(b_3)$ arbitrarily while preserving $\sum_{j=1}^3 \im(\ol{p(a_j)}p(b_j))$ and remain in the $U_A$-orbit closure of $p$. Following the iterative argument from the end of Case 1, we get that we can perturb the real parts of all of the $p(a_j),p(b_j)$ while preserving $\vol(p)$ and remain in the $\Sp(H_1)$-orbit closure of $p$. To show that $\Sp(H_1) \cdot p$ is dense in $V_1^{\R + i\Z}$, it is now enough to show that every tuple $(m_1,\dots,m_g,n_1,\dots,n_g) \in \Z^{2g}$ with $\gcd(m_j,n_j)_{j=1}^g = 1$ arises from the imaginary parts of $p^\pr(a_j),p^\pr(b_j)$ for some $p^\pr \in \Sp(H_1) \cdot p$. This latter statement follows from the fact that $\Sp(2g,\Z)$ acts transitively on primitive vectors in $\Z^{2g}$, thus $\Sp(H_1) \cdot p$ is dense in $V_1^{\R + i\Z}$.
\end{proof}

\paragraph{\bf Navigating $U_A$-orbits within intersections of half-spaces.} We conclude this section with an important technical lemma that will allow us to use the previous bootstrapping lemmas from this section in later arguments about closures of absolute period leaves.

To set up, suppose $g \geq 3$, and fix $z = (z_1,\dots,z_g) \in \C^g$ such that the components $z_1,\dots,z_g$ are pairwise nonparallel. For $\rho = (\rho_1,\dots,\rho_g) \in \R_{> 0}^g$, define
\be
W(z)_\rho = \{(w_1,\dots,w_g) \in W(z) : \im(\ol{z_j} w_j) > \rho_j \text{ for } 1 \leq j \leq g \} .
\ee
Recall that $U_A$ acts on $W(z)$ by shears $s_j : w \ra w + z_j e_j$, $1 \leq j \leq g$, and factor mixes $m_{jk} : w \ra w + z_k e_j + z_j e_k$, $1 \leq j < k \leq g$. Relabel these generators by $u_1,\dots,u_n$, where $n = g(g+1)/2$.

\begin{lem} \label{lem:UAhalfspaces}
Suppose there are $\eps > 0$ and $\rho = (\rho_1,\dots,\rho_g) \in \R_{>0}^g$ with $0 < \eps < \min_j \rho_j$, such that $|\im(\ol{z}_j z_k)| < \eps$ for all $1 \leq j < k \leq g$. Define $\rho^\pr = (\rho_1 - \eps, \dots, \rho_g - \eps) \in \R_{>0}^g$. For any $w_0 \in W(z)_\rho$ and any $u \in U_A$ such that $u \cdot w_0 \in W(z)_\rho$, there is a sequence $u_{j_1}^{s_{j_1}},\dots,u_{j_N}^{s_{j_N}}$ with $s_{j_k} \in \{\pm 1\}$ for $1 \leq k \leq N$, such that the sequence $w_k = u_{j_k}^{s_{j_k}} \cdot w_{k-1}$, $1 \leq k \leq N$, satisfies $w_N = u \cdot w_0$ and $w_k \in W(z)_{\rho^\pr}$ for all $1 \leq k \leq N$.
\end{lem}

In other words, Lemma \ref{lem:UAhalfspaces} says that if the signed areas $\im(\ol{z}_j z_k)$ are sufficiently small, then we can connect any two elements in the same $U_A$-orbit in the product of half-spaces $W(z)_\rho$ by applying a sequence of generators of $U_A$ and their inverses, while staying within a slightly larger product of half-spaces $W(z)_{\rho^\pr}$.

\begin{proof}
Since $U_A$ is free abelian, there are integers $c_j$, $1 \leq j \leq g$, and integers $c_{jk}$, $1 \leq j < k \leq g$, such that
\be
u \cdot w_0 = \prod_{1 \leq j \leq g} s_j^{c_j} \prod_{1 \leq j < k \leq g} m_{jk}^{c_{jk}} \cdot w_0 .
\ee
For $1 \leq j \leq g$ and $w \in W(z)$, define $A_j(w) = \im(\ol{z}_j w_j)$ where $w_j$ is the $j$-th component of $w$. By definition, $w \in W(z)_\rho$ if and only if $A_j(w) > \rho_j$ for $1 \leq j \leq g$ and $\sum_{j=1}^g A_j(w) = 1$. The shears $s_j$ do not change any of the $A_k(w)$, so by replacing $w_0$ with $\prod_{1 \leq j \leq g} s_j^{c_j} \cdot w_0$, we may assume $c_j = 0$ for $1 \leq j \leq g$. The factor mixes $m_{jk}$ only change $A_j(w)$ by $\im(\ol{z}_j z_k)$ and $A_k(w)$ by $-\im(\ol{z}_j z_k)$, and these changes have absolute value less than $\eps$. Thus, applying any single factor mix $m_{jk}^{\pm 1}$ to $w_0$ is guaranteed to stay within $W(z)_{\rho^\pr}$.

We may now assume that $w_0 \in W(z)_{\rho^\pr} \sm W(z)_\rho$ and that
\be
u \cdot w_0 = \prod_{1 \leq j < k \leq g} m_{jk}^{c_{jk}} \cdot w_0 \in W(z)_\rho .
\ee
We will argue by induction on $M = \sum_{1 \leq j < k \leq g} |c_{jk}|$. If $M = 0$ then $w_0 = u \cdot w_0$ and there is nothing to prove. Suppose $M > 0$. Since $w_0 \notin W(z)_\rho$, the set
\be
J = \{1 \leq j \leq g : A_j(w_0) \leq \rho_j\}
\ee
is nonempty, and since $u \cdot w_0 \in W(z)_\rho$, we have $A_j(w_0) < A_j(u \cdot w_0)$ for all $j \in J$. Consider the set of ``pending'' factor mixes
\be
\{m_{jk}^{\sigma_{jk}} : 1 \leq j < k \leq g, \; \sigma_{jk} = \pm 1, \; c_{jk} \cdot \sigma_{jk} > 0\} .
\ee
These are the factor mixes (with powers $\pm 1$) that when applied to $w_0$ decrease $M$ by $1$. Each factor mix decreases some $A_j(w_0)$ and increases some $A_k(w_0)$ by the same amount. Since $\sum_{j \in J} A_j(w_0) < \sum_{j \in J} A_j(u \cdot w_0)$, at least one of these pending factor mixes must increase some $A_j(w_0)$ with $j \in J$, and decrease some $A_k(w_0)$ with $k \notin J$. Applying this pending factor mix decreases $M$ to $M - 1$, increases some $A_j(w_0) \in (\rho_j - \eps, \rho_j]$, decreases some $A_k(w_0) > \rho_k$ by at most $\eps$, and leaves the other $A_i(w_0)$ fixed. Thus, $m_{jk}^{\sigma_{jk}} \cdot w_0 \in W(z)_{\rho^\pr}$ and we are done by induction.
\end{proof}


\section{Absolute Period Foliations and Saddle Connection Configurations} \label{sec:configurations}

We gather background material on strata of holomorphic $1$-forms and their absolute period foliations. We refer to \cite{Zor:flat} for further background material. We then prove some preliminary results about configurations of saddle connections on holomorphic $1$-forms in the strata we will focus on in the next section. \\

\paragraph{\bf Strata of holomorphic $1$-forms.} Let $(X,\om)$ denote a closed Riemann surface $X$ with a holomorphic $1$-form $\om \neq 0$. The set of zeros $Z(\om)$ is finite, and the orders of the zeros form a partition of $2g - 2$, where $g$ is the genus of $X$. Integrating $\om$ away from its zero set provides an atlas of charts into $\C$ whose transition maps are translations. In particular, $\om$ induces a singular flat metric on $X$ with cone points of angle $2\pi(k+1)$ at zeros of order $k \geq 0$. An oriented geodesic segment $\gam$ on $(X,\om)$ is a {\em saddle connection} if it starts and ends in $Z(\om)$ and its interior is contained in $X \sm Z(\om)$, and the {\em holonomy} of $\gam$ is the complex number $\int_\gam \om$. We let $\ol{\gam}$ denote the saddle connection obtained by reversing the orientation of $\gam$, and similarly for oriented closed geodesics. A {\em cylinder} on $(X,\om)$ is a maximal connected open subset of $X$ foliated by parallel closed geodesics. The boundary of a cylinder is a finite union of parallel saddle connections.

Fix an integer $g \geq 2$, and let $\kap = (k_1,\dots,k_n)$ be an unordered partition of $2g - 2$. The {\em stratum} $\Om\cM_g(\kap)$ is the space of all holomorphic $1$-forms on closed genus $g$ Riemann surfaces with exactly $n$ distinct zeros whose orders are given by $\kap$. Strata are complex orbifolds of dimension $2g + n - 1$, and also complex algebraic varieties. Local {\em period coordinates} on a neighborhood in $\Om\cM_g(\kap)$ are given by integrating $1$-forms over a basis for the relative homology group $H_1(X,Z(\om);\Z)$.

Strata can have up to $3$ connected components, and these components were classified in \cite{KZ:components} based on a spin invariant and hyperellipticity. For our purposes, we will only need the following corollary of their classification.

\begin{cor} \label{cor:KZ} (\cite{KZ:components}, Theorems 1-2)
Suppose $g \geq 3$. The stratum $\Om\cM_g(\kap)$ is connected if and only if $\kap$ has at least one odd part and $\kap \neq (g - 1, g - 1)$. Moreover, if $g - 1$ is odd, the stratum $\Om\cM_g(g - 1, g - 1)$ has a unique nonhyperelliptic component.
\end{cor}

The stratum $\Om\cM_g(\kap)$ admits a holomorphic {\em absolute period foliation}, whose leaves are obtained by varying a $1$-form $(X,\om)$ while keeping the integrals $\int_\gam \om$ fixed for all closed loops $\gam$. Absolute period leaves are immersed complex suborbifolds of complex dimension $n - 1$. When $n = 1$, leaves are just points. For any holomorphic $1$-form $(X,\om)$, we denote the {\em absolute periods} of $\om$ by $\Per(\om)$, that is, the set of integrals $\int_\gam \om$ where $\gam$ is a closed loop on $X$. We denote by $\Lam(\om)$ the closure of $\Per(\om)$ in $\C$. Note that $\Per(\om) \subset \C$ is an additive subgroup, $\Lam(\om) \subset \C$ is a closed additive subgroup, and both subgroups are constant along absolute period leaves.

The {\em area} of $(X,\om)$ is given by $\Area(X,\om) = \sum_{j=1}^g \im\left(\int_{a_j} \ol{\om} \int_{b_j} \om\right)$, where $\{a_j,b_j\}_{j=1}^g$ is any symplectic basis for $H_1(X;\Z)$. Note that area is constant along absolute period leaves. For $a > 0$, define
\be
\Om_a\cM_g(\kap) = \{(X,\om) \in \Om\cM_g(\kap) : \Area(X,\om) = a\} .
\ee
Since $\Area(X,\om) > 0$, the absolute periods $\Per(\om)$ always contain a lattice in $\C$. For $\Lam \subset \C$ a closed subgroup containing a lattice, we additionally define
\be
\Om_a^\Lam\cM_g(\kap) = \{(X,\om) \in \Om_a\cM_g(\kap) : \Lam(\om) \subset \Lam \text{ and meets every component of } \Lam\} .
\ee
The subspaces $\Om_a^\Lam\cM_g(\kap) \subset \Om\cM_g(\kap)$ are closed real-analytic suborbifolds.

Let $\GL^+(2,\R)$ be the group of real $2$-by-$2$ matrices with positive determinant. The $\R$-linear action of $\GL^+(2,\R)$ on $\C$ induces an action on strata, by postcomposition with charts on $X \sm Z(\om)$ arising from integration of $\om$. Orbit closures for the $\GL^+(2,\R)$-action on $\Om\cM_g(\kap)$ are complex algebraic subvarieties locally defined by real homogeneous linear equations in period coordinates \cite{EMM:closures}, \cite{Fil:splitting}. Following \cite{AW:highrank}, we will refer to $\GL^+(2,\R)$-orbit closures in strata as {\em invariant subvarieties}. \\

\paragraph{\bf Connected sums and bubbled handles.} We review two standard constructions for producing higher-genus holomorphic $1$-forms from lower genus ones, with a focus on constructing families that lie on a single absolute period leaf. We refer to \cite{EMZ:principal} and \cite{KZ:components} for more general constructions and discussions.

First, we discuss connected sums. Fix an integer $m \geq 2$. Let $(X_1,\om_1),\dots,(X_m,\om_m)$ be a cyclically ordered collection of holomorphic $1$-forms. Choose a nonzero $z \in \C$, and for each $1 \leq j \leq m$, choose an embedded oriented geodesic segment $\gam_j$ in $(X_j,\om_j)$ such that $z = \int_{\gam_j} \om_j$, so the segments $\gam_j$ are all parallel and of the same length. Note that the orientation of $\gam_j$ determines a left and right side of $\gam_j$. Let $p_j,q_j \in X_j$ be the starting point and ending point of $\gam_j$, respectively. The connected sum construction goes as follows. Slit each $X_j$ along $\gam_j$, and then glue the left side of $\gam_j$ to the right side of $\gam_{j+1}$, indices taken modulo $m$. The result is a (connected) holomorphic $1$-form $(X,\om)$ of genus $g = \sum_{j=1}^m g_j$, where $g_j$ is the genus of $X_j$. The points $p_j$ are identified to a single point $p \in X$, and $p \in Z(\om)$ is a zero of order $(m - 1) + \sum_{j=1}^m k_j$, where $k_j$ is the order of $\om_j$ at $p_j$. Similarly for the points $q_j \in X_j$ and the resulting point $q \in X$.

We will be interested in families of connected sums where the starting $1$-forms $(X_j,\om_j)$ are fixed, while the slits vary. Choose a continuous path $z(t) \in \C \sm \{0\}$, $t \in [0,1]$, and choose segments $\gam_j$ on $(X_j,\om_j)$ as above with $z(0) = \int_{\gam_j} \om_j$. For each $j$, the path $z(t)$ determines a continuous path of directions $\theta_j(t)$ emanating from the starting point $p_j$. We assume that for all $t \in [0,1]$ and all $j$, the $1$-form $(X_j,\om_j)$ has an embedded oriented geodesic segment $\gam_j(t)$ starting at $p_j$ in the direction of $\theta_j(t)$ with holonomy $z(t)$, and that the interior of $\gam_j(t)$ lies in $X_j \sm Z(\om_j)$. Then the connected sum construction produces a continuous path of $1$-forms $(X(t),\om(t))$. Moreover, since the integrals of $\om_j$ along closed loops in $X_j$ remain constant, and since $H_1(X(t);\Z) \cong \bigoplus_{j=1}^m H_1(X_j;\Z)$, the path $(X(t),\om(t))$ lies in an absolute period leaf.

Next, we discuss bubbled handles. Let $(X,\om)$ be a holomorphic $1$-form of genus $g$, and let $T = (\C/(\Z z_1 + \Z w_1), dz)$ be a flat torus. Let $\al \subset T$ be a closed geodesic with $z_1 = \int_\al dz$, and let $\gam$ be an embedded oriented geodesic segment in $(X,\om)$ with $z_1 = \int_\gam \om$ starting at $Z \in Z(\om)$ and ending at a point in $X \sm Z(\om)$. The bubbled handle construction goes as follows. Slit $X$ along $\gam$ and slit $T$ along $\al$. Then glue the left side of $\gam$ to the right side of $\al$, and glue the right side of $\gam$ to the left side of $\al$. The result is a holomorphic $1$-form $(X^\pr,\om^\pr)$ of genus $g + 1$. If $\om$ has a zero of order $k$ at $Z$, then the corresponding point on $(X^\pr,\om^\pr)$ is a zero of $\om^\pr$ of order $k + 2$.

We will be interested in the following families of bubbled handles. Suppose $(X(t),\om(t))$, $t \in [0,1]$, is a continuous path in an absolute period leaf. Choose a flat torus $T$ and a segment $\gam$ on $(X(0),\om(0))$ as above. The choice of $\gam$ determines a continuous path of directions $\theta(t)$ emanating from $Z$ on $(X(t),\om(t))$. Suppose that for all $t \in [0,1]$, the $1$-form $(X(t),\om(t))$ does not have a saddle connection starting at $Z$ in the direction of $\theta(t)$ with holonomy in $[0,1] \cdot z_1$. Then the bubbled handle construction yields a well-defined path $(X^\pr(t),\om^\pr(t))$ in another absolute period leaf (in a higher genus stratum). \\

\paragraph{\bf Multiple connected sum presentations.} We now present a criterion in terms of simple inequalities in certain periods for the existence of two presentations of a single $1$-form as a connected sum, where the associated saddle connections are related by a Dehn twist. Similar criteria are studied in \cite{CM:non-ergodic} and \cite{McM:SL2R}.

For nonzero complex numbers $z,w$ with $\im(\ol{z} w) > 0$, let
\be
P(z,w) = \{t_1 z + t_2 w : 0 < t_1,t_2 < 1\}, \quad \ol{P}(z,w) = \{t_1 z + t_2 w : 0 \leq t_1,t_2 \leq 1\}
\ee
be the open (resp. closed) parallelograms in $\C$ with vertices $0,z,z+w,w$ in counterclockwise order. Suppose $(X,\om)$ is a holomorphic $1$-form with an embedded open parallelogram $P$ bounded by saddle connections $\gam_1,\gam_2,\gam_3,\gam_4$ in counterclockwise order around $\partial P$. The saddle connections $\gam_j$ are not necessarily distinct. Letting $z = \int_{\gam_1} \om$, $w = -\int_{\gam_4} \om$, we have that $\im(\ol{z} w) = \Area(P) > 0$, and the identification $P(z,w) \cong P$ extends to a continuous surjection $\ol{P}(z,w) \ra \ol{P}$. For $0 \leq t_1,t_2 \leq 1$, let $P(t_1 z + t_2 w) \in \ol{P}$ be the image of $t_1 z + t_2 w$ under this surjection.

\begin{lem} \label{lem:sc12}
Let $(X,\om)$ be a holomorphic $1$-form with a pair of homologous saddle connections $\gam,\gam^\pr$. Suppose that for $j = 1,2$, we have a pair $\al_j,\al_j^\pr$ of closed geodesics on $(X,\om)$ such that $\al_j,\gam^\pr,\ol{\al}_j^\pr,\ol{\gam}$ bound an embedded open parallelogram $P_j$ and are in counterclockwise order around $\partial P_j$. Suppose also that $P_1$ and $P_2$ are disjoint, and that $\ol{P}_1 \cap \ol{P}_2 = \gam \cup \gam^\pr$. Let $c = \int_\gam \om = \int_{\gam^\pr} \om$, and for $j = 1,2$, let $z_j = \int_{\al_j} \om = \int_{\al_j^\pr} \om$. If
\be
-\im(\ol{z}_1 c) < \im(\ol{z}_1 z_2) < \im(\ol{z}_2 c) ,
\ee
then there is another pair of homologous saddle connections $\del,\del^\pr$ in $\ol{P}_1 \cup \ol{P}_2$ with
\be
z_1 + z_2 + c = \int_\del \om = \int_{\del^\pr} \om
\ee
and isotopic to a Dehn twist of $\gam,\gam^\pr$ along a simple closed piecewise-geodesic curve $\al_{12}$ in $\ol{P}_1 \cup \ol{P}_2$.
\end{lem}

\begin{proof}
Since
\begin{align*}
\im(\ol{z}_1 (z_1 + z_2 + c)) &= \im(\ol{z}_1 z_2) + \im(\ol{z}_1 c) > 0 \\
\im((\ol{z_1 + z_2 + w}) (z_1 + c)) &= -\im(\ol{z}_1 z_2) + \im(\ol{z}_2 c) > 0 \\
\end{align*}
there is a geodesic segment $s_1 \subset \ol{P}_1$ in the direction of $z_1 + z_2 + c$, starting at the common starting point of $\gam$ and $\al_1$, and ending at a point in the interior of $\gam^\pr$. Similarly, since
\begin{align*}
\im(\ol{(z_1 + z_2 + c)} (z_2 + c)) &= \im(\ol{z}_1 z_2) + \im(\ol{z}_1 c) > 0 \\
\im(\ol{z}_2 (z_1 + z_2 + c)) &= -\im(\ol{z}_1 z_2) + \im(\ol{z}_2 c) > 0
\end{align*}
there is a geodesic segment $s_2 \subset \ol{P}_2$ in the direction of $z_1 + z_2 + c$, starting at a point in the interior of $\gam^\pr$, and ending at the common ending point of $\gam$ and $\al_2^\pr$. The segment $s_1$ ends at a point of the form $P_1(z_1 + t_2 c)$, and the segment $s_2$ starts at a point of the form $P_2(t_1 c)$. Since
\be
\im((\ol{z_1 + t_2 c}) (z_1 + z_2 + c)) = 0, \quad \im((\ol{z_2 + (1 - t_1)c}) (z_1 + z_2 + c)) = 0,
\ee
solving for $t_1$ and $t_2$ gives us
\be
t_1 = \frac{\im(\ol{z}_1 z_2) + \im(\ol{z}_1 c)}{\im(\ol{z}_1 c) + \im(\ol{z}_2 c)} = t_2 ,
\ee
and the above inequalities on $\im(\ol{z}_1 z_2)$ and $\im(\ol{z}_j c)$ ensure $0 < t_1,t_2 < 1$. Thus, the union $s_1 \cup s_2$ yields a saddle connection $\del$ in $\ol{P}_1 \cup \ol{P}_2$ with $z_1 + z_2 + c = \int_\del \om$. The same argument with the roles of $P_1$ and $P_2$ reversed yields a second saddle connection $\del^\pr$ in $\ol{P}_1 \cup \ol{P}_2$ with $z_1 + z_2 + c = \int_{\del^\pr} \om$.

Lastly, let $\al_{12} \subset \ol{P}_1 \cup \ol{P}_2$ be the closed loop obtained from the two geodesic segments $P_1(t_1 z_1 + c / 2)$, $0 \leq t_1 \leq 1$, and $P_2(t_2 z_2 + c / 2)$, $0 \leq t_2 \leq 1$. Then $\del,\del^\pr$ are isotopic to the Dehn twist of $\gam,\gam^\pr$ along the loop $\al_{12}$.
\end{proof}


\section{Leaf Closures in Connected Strata} \label{sec:leafclosures}

In this section, we prove Theorem \ref{thm:leafclosures}. The general approach will be to study the recurrence properties for absolute period leaves in a certain open neighborhood in period coordinates arising from a more complicated configuration of saddle connections, in order to apply the results from Section \ref{sec:Kapovich}. \\

\paragraph{\bf Antler configurations.} Fix an integer $g \geq 3$. Let $(X,\om)$ be a holomorphic $1$-form of genus $g \geq 3$ with exactly two zeros $Z_1,Z_2$ of odd orders $k_1,k_2$, respectively. An {\em antler configuration} of saddle connections on $(X,\om)$ is a collection of pairs of homologous saddle connections $\gam,\gam^\pr$, and $\al_j,\al_j^\pr$, $3 \leq j \leq g$, with the following properties.

\begin{itemize}
    \item The pair $\gam,\gam^\pr$ starts at $Z_1$ and ends at $Z_2$. For $3 \leq j \leq (k_1 + 3)/2$, the pair $\al_j,\al_j^\pr$ starts and ends at $Z_1$. For $(k_1 + 5)/2 \leq j \leq g$, the pair $\al_j,\al_j^\pr$ starts and ends at $Z_2$.
    \item One of the components of $X \sm (\gam \cup \gam^\pr)$ is a flat torus $T_2$ with a slit removed. For $3 \leq j \leq g$, one of the components of $X \sm (\al_j \cup \al_j^\pr)$ is a cylinder $C_j$.
    \item Let $S_1$ be the union of $\gam,\gam^\pr,\al_3,\al_3^\pr,\dots,\al_g,\al_g^\pr$. Let $S_2$ be the union of $T_2,C_3,\dots,C_g$. The complement $X \sm (S_1 \cup S_2)$ is a flat torus $T_1$ with a finite union of segments removed, and the removed segments form an embedded geodesic tree in $T_1$.
    \item For $3 \leq j \leq g$, the counterclockwise angle from the end of $\al_j$ to the end of $\al_j^\pr$ is $2\pi$.
    \item Around $Z_1$, the starts of $\gam,\al_3,\dots,\al_{(k_1 + 3)/2},\gam^\pr$ are in counterclockwise order. Around $Z_2$, the starts of $\ol{\gam}^\pr,\al_{(k_1 + 5)/2},\dots,\al_g,\ol{\gam}$ are in counterclockwise order.
\end{itemize}

We refer to $T_1,T_2$ as the {\em associated flat tori} and $C_3,\dots,C_g$ as the {\em associated cylinders} of the antler configuration. We refer to $\int_\gam \om$ as the {\em slit parameter}. Let $\del_j$ be a saddle connection in $C_j \cup Z(\om)$ crossing $C_j$ from bottom to top with respect to the orientation of $\al_j$. We refer to $\int_{\al_j} \om, \int_{\del_j} \om$ as the {\em handle parameters} for $C_j$. \\

Fix an integer $g \geq 3$, and fix odd integers $k_1 \geq k_2$ such that $k_1 + k_2 = 2g - 2$. Fix $\eps = \eps(g) > 0$ small, and fix $\eps_1 = \eps_1(g) > 0$ such that $\eps_1 / \eps$ is small. Let $n_1 = (k_1 - 1)/2$ and $n_2 = (k_2 - 1)/2$. Let $u_1,u_2,r_1,r_2$ all be $1$. Let $u_3,\dots,u_{(k_1 + 3)/2}$ be the complex arithmetic progression
\be
\frac{1}{10} - n_1 \eps i, \; \frac{(n_1-1)}{10n_1} - (n_1 - 1)\eps i, \; \dots, \; \frac{1}{10n_1} - \eps i,
\ee
and let $r_3,\dots,r_{(k_1+3)/2}$ be the real arithmetic progression $1/10, (n_1 - 1)/10n_1, \dots, 1/10n_1$. Let $u_{(k_1 + 5)/2},\dots,u_g$ be the complex arithmetic progression
\be
-\eps + \frac{1}{10}i, \; -2\eps + \frac{(n_2-1)}{10n_2}i, \; \dots, \; -n_2 \eps + \frac{1}{10n_2}i ,
\ee
and let $r_{(k_1+5)/2},\dots,r_g$ be the real arithmetic progression $1/10, (n_2-1)/10n_2, \dots, 1/10n_2$. Lastly, let $c_0 = -\frac{1}{10} + \frac{1}{10}i$. Note that
\be
r_j > \frac{1}{10g}, \quad 1 \leq j \leq g .
\ee

For our convenience, we will denote
\be
z = (z_1,\dots,z_g), \quad w = (w_1,\dots,w_g), \quad (z,w) = (z_1,\dots,z_g,w_1,\dots,w_g) .
\ee
Let $e_1,\dots,e_g$ be the standard basis vectors in $\C^g$. Define
\begin{align*}
\mc{U}(k_1,k_2) &= \{(z,w) \in \C^{2g} : |z_j - u_j| < \eps_1, \im(\ol{z}_j w_j) > r_j + \frac{1}{100g}, 1 \leq j \leq g\} \\
\mc{U}_1(k_1,k_2) &= \{(z,w,c) \in \C^{2g+1} : (z,w) \in \mc{U}(k_1,k_2), |c - c_0| < \eps_1 \} .
\end{align*}
For all $(z,w,c) \in \mc{U}_1(k_1,k_2)$, and for $1 \leq j < k \leq (k_1 + 3)/2$, since $z_j,z_k$ have real parts in $(0,1 + \eps)$ and imaginary parts in $(-g\eps,g\eps)$, we have
\begin{equation} \label{eq:zjkU}
\left| \im(\ol{z}_j z_k) \right| < \im((\ol{1 - g\eps i})(1 + g\eps i)) < 3 g \eps
\end{equation}
since $\eps = \eps(g)$ is small. Similarly for $w_1,w_2,z_{(k_1 + 5)/2},\dots,z_g$. For $1 \leq j \leq g$, we similarly have 
\begin{equation} \label{eq:zjcU}
\frac{1}{200g} < \im(\ol{z}_j c) < \frac{1}{10} + \eps .
\end{equation}

If $k_1 > k_2$, let $\mc{C}$ be the connected stratum $\Om\cM_g(k_1,k_2)$. If $k_1 = k_2$, let $\mc{C}$ be the nonhyperelliptic component of $\Om\cM_g(g-1,g-1)$. Each $(z,w,c) \in \mc{U}(k_1,k_2)$ determines a well-defined $1$-form in $\mc{C}$ with an antler configuration with associated flat tori $T_1 = (\C / (\Z z_1 + \Z w_1), dz)$, $T_2 = (\C / (\Z z_2 + \Z w_2), dz)$, slit parameter $c$ for the slits $\gam,\gam^\pr$, handle parameters $z_j,w_j$ for the associated cylinder $C_j$, $3 \leq j \leq g$. In this way, we get an {\em antler map}
\be
\Psi_\cC : \mc{U}_1(k_1,k_2) \ra \cC ,
\ee
and an associated {\em leaf map}
\be
L_\cC : \mc{U}(k_1,k_2) \ra \{\text{absolute period leaves in } \cC\}
\ee
that sends $(z,w)$ to the absolute period leaf through $\Psi_\cC(z,w,c)$. The resulting leaf does not depend on the choice of $c$ such that $|c - c_0| < \eps_1$.

\begin{lem} \label{lem:dehntwistrecurrence}
For all $(z,w) \in \mc{U}(k_1,k_2)$, we have the following.
\begin{enumerate}
    \item $L_\cC(z,w) = L_\cC(z, w + z_j e_j)$ for all $1 \leq j \leq g$.
    \item $L_\cC(z,w) = L_\cC(z, w + z_2 e_1 + z_1 e_2)$.
    \item $L_\cC(z,w) = L_\cC(z, w + z_k e_j + z_j e_k)$ for all $1 \leq j \leq 2$ and $3 \leq k \leq (k_1 + 3)/2$.
    \item $L_\cC(z,w) = L_\cC(z + w_2 e_1 + w_1 e_2, w)$.
    \item $L_\cC(z,w) = L_\cC(z + z_k e_j, w + w_j e_k)$ for all $1 \leq j \leq 2$ and $(k_1 + 5)/2 \leq k \leq g$.
\end{enumerate}
\end{lem}

\begin{proof}
First, we prove item (1). For $1 \leq j \leq g$, the holomorphic $1$-forms $\Psi_\cC(z,w,c)$ and $\Psi_\cC(z, w + z_j e_j, c)$ are the same, and in particular $L_\cC(z,w) = L_\cC(z, w + z_j e_j)$.

Next, we prove item (2). Let $(X,\om) = \Psi_\cC(z,w,c)$. On $T_2$, there is an embedded open parallelogram $P_2$ with one pair of parallel sides given by the slits $\gam,\gam^\pr$ and the other pair given by two closed geodesics with period $z_2$. By definition of $\cU_1(k_1,k_2)$, we have $-3g\eps < \im(\ol{z}_1 z_j) < 0$ for $3 \leq j \leq (k_1 + 3)/2$, $0 < \im(\ol{z}_1 z_j) < \frac{1}{10} + \eps$ for $(k_1 + 5)/2 \leq j \leq g$, $0 < \im(\ol{z}_1 c) < 1/10 + \eps$, and $\im(\ol{z}_1 w_1) > 9/10$. Letting $M_1 = \max_{3 \leq j \leq (k_1 + 3)/2} |\im(\ol{z}_1 z_j)|$ and $M_2 = \max_{(k_1 + 5)/2 \leq j \leq g} |\im(\ol{z}_1 z_j)|$, we then have
\be
\im(\ol{z}_1 w_1) > \frac{9}{10} > \im(\ol{z}_1 c) + M_1 + M_2 .
\ee
Thus, there are $3$ disjoint open bands $B_1,B_2,B_3 \subset T_1$ of closed geodesics parallel to $z_1$ such that the interiors of the slits $\gam,\gam^\pr$ lie in $B_1$, the interiors of the slits $\al_3,\dots,\al_{(k_1+3)/2}$ lie in $B_2$, and the interiors of the slits $\al_{(k_1+5)/2},\dots,\al_g$ lie in $B_g$. In particular, $B_1 \subset T_1$ contains an embedded open parallelogram $P_1$ with one pair of parallel sides given by $\gam,\gam^\pr$ and the other pair given by two closed geodesics with period $z_1$. Since $|\im(\ol{z}_j z_k)| < 3g \eps$ and $\im(\ol{z}_j c) > 1/200g$ and $\eps = \eps(g)$ is small, we can apply Lemma \ref{lem:sc12} to get another pair of homologous saddle connections $\del,\del^\pr \subset \ol{P}_1 \cup \ol{P}_2$ with $\int_\del \om = c + z_1 + z_2$, obtained by applying a Dehn twist in $\ol{P}_1 \cup \ol{P}_2$ to $\gam,\gam^\pr$. Moreover, we similarly have
\be
\im(\ol{z}_1 (w_1 + z_1 + z_2)) > \frac{9}{10} > \im(\ol{z}_1 (c + z_1 + z_2)) + M_1 + M_2 ,
\ee
so $(X,\om)$ admits a second antler configuration, obtained by applying the same Dehn twist. The associated flat tori are
\be
T_1^\pr = (\C / (\Z z_1 + \Z (w_1 + z_2))), \quad T_2^\pr = (\C / (\Z z_2 + \Z (w_2 + z_1))),
\ee
the slit parameter is $c + z_1 + z_2$, and the handle parameters remain the same. Finally, along the line segment $c(t) = tc + (1-t)(c + z_1 + z_2) \in \C \sm \{0\}$, $t \in [0,1]$, from $c + z_1 + z_2$ to $c$, we similarly have
\be
\im(\ol{z}_1 c(t)) + M_1 + M_2 < \frac{1}{5} + 10 g \eps < \frac{9}{10} ,
\ee
and so the path $c(t) \in \C \sm \{0\}$ determines a well-defined path of connected sums in the absolute period leaf through $(X,\om)$, starting at $(X,\om)$ and ending at $\Psi_\cC(z,w + z_2 e_1 + z_1 e_2,c)$. Thus, $L_\cC(z,w) = L_\cC(z,w + z_2 e_1 + z_1 e_2)$.

The proof of item (4) is similar to item (2), with $w_1,w_2$ in place of $z_1,z_2$.

Next, we prove item (3). Suppose $j = 2$ and  and $3 \leq k \leq (k_1 + 3)/2$. Start with $(X,\om) = \Psi(z,w,c)$, and consider the path $c(t) \in \C$ given by $c(0) = c$, $c(t+1) = c(t)$, and
\be
c(t) = \begin{cases}
    c(0) - \frac{4t}{5}i, \quad t \in [0,1/4] \\
    c(1) + 4(t - 1/4)\left(\frac{1}{10} + 2g\eps\right) , \quad t \in [1/4,1/2] \\
    c(2) + \frac{4(t - 1/2)}{5}i, \quad t \in [1/2,3/4] \\
    c(3) - 4(t - 3/4)\left(\frac{1}{10} + 2g\eps\right), \quad t \in [3/4,1] .
\end{cases}
\ee
The path $c(t)$ determines a path $(X_t,\om_t) \in L_\cC(z,w)$ on which the holonomy of the slit $\gam$ is $c(t)$. Roughly speaking, along the path $(X_t,\om_t)$, the slits $\gam,\gam^\pr$ ``rotate'' counterclockwise by $2\pi$ around their starting point. The particular rectilinear path above is chosen to ensure that the saddle connections $\al_j,\al_j^\pr$, $(k_1 + 5)/2,\dots,g$, persist along this path. Let $\theta_0$ be the direction of $\gam$ emanating from its starting point $Z_1$, and for $n \in \Z$, let $\theta_n$ be obtained from $\theta_0$ by rotating counterclockwise around $Z_1$ by $2\pi n$. In particular, the directions $\theta_1,\theta_3,\dots,\theta_{k_1 - 2}$ lie in the cylinders $C_3,C_4,\dots,C_{(k_1+3)/2}$. Then on $(X_{2j-1},\om_{2j-1})$, there is an embedded open parallelogram $P_{j+2}$ with one pair of parallel sides given by $\gam,\gam^\pr$ and the other pair given by two closed geodesics with period $z_j$.  On $T_2 \subset X_{2j-1}$, there is another embedded open parallelogram $P_{j+2}^\pr$ bounded by $\gam,\gam^\pr$ and a two closed geodesics with period $z_2$. Then as in item (2), by equations (\ref{eq:zjkU}) and (\ref{eq:zjcU}) we can apply Lemma \ref{lem:sc12} to get another pair of homologous saddle connections $\del,\del^\pr \subset \ol{P}_{j+2} \cup \ol{P}^\pr_{j+2}$ obtained by applying a Dehn twist in $\ol{P}_{j+2} \cup \ol{P}^\pr_{j+2}$ to $\gam,\gam^\pr$. Thus, $(X_{2j-1},\om_{2j-1})$ can also be presented by replacing $\gam,\gam^\pr$ with $\del,\del^\pr$, replacing $C_j$ with a cylinder with handle parameters $z_j,w_j + z_2$, and replacing $T_2$ with the flat torus $T_{2,j} = (\C / (\Z z_2 + \Z (w_2 + z_j)), dz)$. As in item (2), the linear path $c_j(t) \in \C$, $t \in [0,1]$, from $c_j(0) = c + z_j + z_2$ to $c_j(1) = c$, determines a well-defined path of connected sums in the absolute period leaf through $(X,\om)$, starting at $(X_{2j-1},\om_{2j-1})$ and ending at a holomorphic $1$-form $(Y,\eta)$. Finally, let $(X^\pr,\om^\pr) = \Psi_\cC(z,w + z_{2j-1} e_2 + z_2 e_{2j-1}, c)$, so the given antler configuration for $(X^\pr,\om^\pr)$ has the same slit parameter as that of $(X,\om)$, the same handle parameters except for a bubbled handle with periods $z_j,w_j + z_2$ instead of $z_j,w_j$, and an associated flat torus $T_2^\pr = (\C / (\Z z_2 + \Z (w_2 + w_{2j-1})), dz)$ instead of $T_2$. The same path $c(t)$ above determines a path $(X_t^\pr,\om_t^\pr) \in L_\cC(z,w + z_{2j-1} e_2 + z_2 e_{2j-1})$, and we have $(X_{2j-1}^\pr,\om_{2j-1}^\pr) = (Y,\eta)$. Thus, $L_\cC(z,w) = L_\cC(z,w + z_{2j-1} e_2 + z_2 e_{2j-1})$.

To prove item (3) when $j = 1$ and $3 \leq k \leq (k_1 + 3)/2$, start with $(X,\om) = \Psi(z,w,c)$ again, and now travel along the path $(X_t,\om_t) \in L(z,w)$ determined by the path $c_1(t) \in \C \sm \{0\}$ given by $c_1(0) = c$ and
\be
c_1(t) = \begin{cases}
    c_1(0) - \frac{4t}{5}i, \quad t \in [0,1/4] \\
    c_1(1) + \frac{4(t - 1/4)}{5} , \quad t \in [1/4,1/2] \\
    c_1(2) + \frac{4(t - 1/2)}{5}i, \quad t \in [1/2,3/4] \\
    c_1(3) - \frac{4(t - 3/4)}{5}, \quad t \in [3/4,1] .
\end{cases}
\ee
The $1$-form $(X_1,\om_1)$ has an antler configuration with the same slit parameter and bubbled handle parameters, but the associated flat tori $T_2,T_1$ have been swapped. The rest of the argument is now the same as in the $j = 1$ case of item (3).

Lastly, the proof of item (5) is similar to item (3), with $w_1,w_2,z_{(k_1 + 5)},\dots,z_g$ in place of $z_1,z_2,z_3,\dots,z_{(k_1+3)/2}$.
\end{proof}

Fix $(X,\om) \in \Psi_\cC(z,w,c)$ given by an antler configuration as above. Choose closed geodesics $\al_j,\del_j$, $1 \leq j \leq 2$, with periods $z_j,w_j$, respectively. Then the closed loops $\{\al_j,\del_j\}_{j=1}^g$ represent a symplectic basis $\{a_j,b_j\}_{j=1}^g$ for $H_1(X;\Z)$. Let $p_\om : H_1(X;\Z) \ra \C$ be the associated homomorphism given by integrating $\om$ over closed loops in $X$.

For $3 \leq j_1 \leq (k_1 + 3)/2$, let $U_{j_1} \subset \Sp(H_1(X;\Z))$ be the free abelian subgroup generated by the shears $s_j$, $j \in \{1,2,j_1\}$, given by $s_j(b_j) = b_j - a_j$, and by the factor mixes $m_{jk}$, $j < k$ and $j,k \in \{1,2,j_1\}$, given by $m_{jk}(b_j) = b_j - a_k$ and $m_{jk}(b_k) = b_k - a_j$. Any basis elements not mentioned are fixed. Similarly, for $(k_1 + 5)/2 \leq j_1 \leq g$, let $U_{j_1}^\pr \subset \Sp(H_1(X;\Z))$ be the free abelian subgroup generated by the shears $s_j$, $j \in \{1,2,j_1\}$, given by $s_j(b_j) = b_j - a_j$, and by the modified factor mixes $m_{jk}^\pr$, $j < k$ and $j,k \in \{1,2,j_1\}$, given by $m_{jk}^\pr(b_j) = b_j - b_k$ and $m_{jk}^\pr(a_k) = a_k - a_j$. In the local period coordinate neighborhood determined by the antler map $\Psi_\cC$, we can identify connected components of absolute period leaves with homomorphisms $H_1(X;\Z) \ra \C$.

For $(X,\om) \in \mc{U}(k_1,k_2)$, we denote by $L_\cC(p_\om)$ the absolute period leaf through $(X,\om)$. For $3 \leq j \leq g$ and $u \in U_j$ such that
\be
((u \cdot p_\om)(a_1),\dots,(u \cdot p_\om)(b_g)) \in \mc{U}(k_1,k_2) ,
\ee
we let $L_\cC(u \cdot p_\pm)$ denote the associated absolute period leaf. We will simply say $u \cdot p_\om \in \mc{U}(k_1,k_2)$ going forward.

For $r > 0$, define a smaller neighborhood $\mc{U}(k_1,k_2,r) \subset \mc{U}(k_1,k_2)$ by changing the inequalities $\im(\ol{z}_j w_j) > r_j + 1/100g$ in the definition of $\mc{U}(k_1,k_2)$ to the inequalities
\be
\im(\ol{z}_j w_j) > r + r_j + \frac{1}{100g}, \quad 1 \leq j \leq g.
\ee

\begin{lem} \label{lem:Uorbitleaf}
Fix $r > 3 g \eps$, fix $(X,\om) \in \mc{U}(k_1,k_2,r)$, and fix $3 \leq j \leq g$. For all $u \in U_j$ such that $u \cdot p_\om \in \mc{U}(k_1,k_2,r)$, we have $L_\cC(p_\om) = L_\cC(u \cdot p_\om)$.
\end{lem}

\begin{proof}
Let $z_k,w_k,c$ be the local period coordinates of $(X,\om)$ arising from the given antler configuration on $(X,\om)$. We may assume $3 \leq j \leq (k_1 + 3)/2$. The case where $(k_1 + 5)/2 \leq j \leq g$ is similar, with $w_1,w_2$ in place of $z_1,z_2$. Moreover, by symmetry we may additionally assume $j = 3$. The action of $U_3$ only changes the periods $w_1,w_2,w_3$ and leaves the other period coordinates $z_1,\dots,z_g,w_4,\dots,w_g,c$ fixed. By definition of $\mc{U}(k_1,k_2,r)$, we have
\be
\{(w_1^\pr,w_2^\pr,w_3^\pr) \in \C^3 : (z_1,\dots,z_g,w_1^\pr,w_2^\pr,w_3^\pr,w_4,\dots,w_g) \in \mc{U}(k_1,k_2,r)\} = W(z^\pr)_\rho
\ee
with $z^\pr = (z_1,z_2,z_3)$ and $\rho = (r + r_1 + 1/100g, r + r_2 + 1/100g, r + r_3 + 1/100g) \in \R_{>0}^3$. By definition of $\mc{U}(k_1,k_2,r)$, we have $\im(\ol{z}_j w_j) > 1/20g$ for $1 \leq j \leq g$, and we have $|\im(\ol{z}_j z_k)| < 3g \eps$ for $1 \leq j < k \leq g$. Let $\eps^\pr = 3g \eps$. Since $\eps = \eps(g)$ is small, we have $0 < \eps^\pr < \min_j \rho_j$. Let $\rho^\pr = (\rho_1 - \eps^\pr, \rho_2 - \eps^\pr, \rho_3 - \eps^\pr)$, and note that
\be
\{(z_1,\dots,z_g,w_1^\pr,w_2^\pr,w_3^\pr,w_4,\dots,w_g) \in \C^{2g} : (w_1^\pr,w_2^\pr,w_3^\pr) \in W(z^\pr)_{\rho^\pr}\} \subset \mc{U}(k_1,k_2) .
\ee
We can apply Lemma \ref{lem:UAhalfspaces} to
\be
(p_\om(w_1),p_\om(w_2),p_\om(w_3)), \; ((u \cdot p_\om)(w_1), (u \cdot p_\om)(w_2), (u \cdot p_\om)(w_3)) \in W(z^\pr)_\rho ,
\ee
to get a sequence $u_0,\dots,u_N \in U_3$ such that $u_0$ is the identity and $u_k u_{k-1}^{-1}$ is one of the shears $s_1^{\pm 1},s_2^{\pm 1},s_3^{\pm 1}$ or one of the factor mixes $m_{12}^{\pm 1},m_{23}^{\pm 1},m_{13}^{\pm 1}$ and such that
\be
((u \cdot p_\om)(w_1), (u \cdot p_\om)(w_2), (u \cdot p_\om)(w_3)) \in W(z^\pr)_{\rho^\pr}
\ee
for $1 \leq k \leq N$. Passing back to elements of $\mc{U}(k_1,k_2)$, we can now inductively apply Lemma \ref{lem:dehntwistrecurrence} to get
\be
L_\cC(p_\om) = L_\cC(u_1 \cdot p_\om) = L_\cC(u_2 \cdot p_\om) = \cdots = L_\cC(u_N \cdot p_\om) = L(u \cdot p_\om) .
\ee
\end{proof}

\begin{cor} \label{cor:leafclosureU}
Fix $r > 3g\eps$, fix $(X,\om) \in \mc{U}(k_1,k_2,r)$, and fix $3 \leq j \leq g$. If $(X^\pr,\om^\pr) \in \mc{U}(k_1,k_2,r)$ is sufficiently nearby to $(X,\om)$, and the local period coordinate vector $(z^\pr,w^\pr) \in \C^{2g}$ of $(X^\pr,\om^\pr)$ lies in the $U_j$-orbit closure of the local period coordinate vector $(z,w) \in \C^{2g}$ of $(X,\om)$, then $(X^\pr,\om^\pr)$ is in the closure of the absolute period leaf $L_\cC(p_\om)$.
\end{cor}

We now shrink the neighborhood $\mc{U}(k_1,k_2)$ further as follows. Define a smaller neighborhood $\mc{U}^\pr(k_1,k_2) \subset \mc{U}(k_1,k_2)$ by changing the inequalities $\im(\ol{z}_j w_j) > r_j + 1/100g$ in the definition of $\mc{U}(k_1,k_2)$ to the inequalities
\begin{align*}
|w_j - i| < \eps_1, & \quad 1 \leq j \leq (k_1 + 3)/2, \\
|w_j + 1| < \eps_1, & \quad (k_1 + 5)/2 \leq j \leq g .
\end{align*}
On $(X,\om) \in \mc{U}^\pr(k_1,k_2)$, the tori $T_1,T_2$ and the cylinders $C_3,\dots,C_{(k_1 + 3)/2}$ each contain an embedded open parallelogram $P_j$ with one pair of parallel sides given by two closed geodesics with period $w_j$. Since $\eps$ and $\eps_1$ are small, these periods $w_j$ are nearly vertical. Thus, we can apply certain Dehn twists in the union of the $\ol{P}_3$ to the pairs $\al_j,\al_j^\pr$ to obtain other antler configurations for $(X,\om)$. For our purposes, we need to ensure that the signed areas $\im(\ol{z}_j z_k)$ remain small on the new antler configuration.

For $1 \leq j \leq 2$, let $\bet_j$ be a closed geodesic in $T_j$ with period $w_j$ that is disjoint from all of the slits on $T_j$. For $3 \leq j \leq (k_1 + 3)/2$, let $\bet_j$ be a closed geodesic in $\ol{P}_1 \cup (\ol{P}_3 \cup \cdots \cup \ol{P}_j)$ with period $w_1 + (w_3 + \cdots + w_j)$. The existence of such $\bet_j$ is ensured because the periods $w_1,w_3,\dots,w_j$ are all very close to $i$. The homology class of $\bet_j$ is $[\bet_j] = b_1 + (b_3 + \cdots + b_j)$. Denote
\be
m = \frac{k_1 + 3}{2}, \quad L = \lcm(2,3,\dots,m-1) .
\ee
Applying the Dehn multitwist
\be
T_{\bet_1}^{9(m-2)L} T_{\bet_2}^{10(m-2)L} T_{\bet_3}^{L/2} \cdots T_{\bet_m}^{L/(m-1)} ,
\ee
we obtain another antler configuration with bubbled handles bounded by pairs $\wt{\al}_j,\wt{\al}_j^\pr$, $3 \leq j \leq m$, with periods $\wt{z}_j$ given by
\begin{align*}
\wt{z}_3 &= z_3 + \frac{L}{2}(w_1 + w_3) + \frac{L}{3}(w_1 + w_3 + w_4) + \cdots + \frac{L}{m-1}(w_1 + w_3 + \cdots + w_m)  \\
\wt{z}_4 &= z_4 + \frac{L}{3}(w_1 + w_3 + w_4) + \cdots + \frac{L}{m-1}(w_1 + w_3 + \cdots + w_m) \\
&\vdots \\
\wt{z}_m &= z_m + \frac{L}{m-1}(w_1 + w_3 + \cdots + w_m)
\end{align*}
and with handle parameters $\wt{z}_3,w_3,\dots,\wt{z}_m,w_m$. The associated flat tori $T_1,T_2$ remain the same, and we note that the image of a closed geodesic $\al_j \subset T_j$ with period $z_j$ is a closed geodesic with period
\begin{align*}
\wt{z}_1 &= z_1 + 9(m - 2)L w_1 + \frac{L}{2}(w_1 + w_3) + \cdots + \frac{L}{m-1}(w_1 + w_3 + \cdots + w_m) \\
\wt{z}_2 &= z_2 + 10(m-2)L w_2
\end{align*}
The absolute values $|\wt{z}_j|$ are bounded in terms of $g$, and the periods $\wt{z}_j$ are all close to the real line through $0$ and $1/(10(m-2)) + Li$, so the signed areas $\im(\ol{\wt{z}_j} \wt{z}_k)$ remain small. It follows that there is an element $M \in \GL^+(2,R)$ such that on $M \cdot (X,\om)$, the antler configuration with bubbled handles $\wt{\al}_j,\wt{\al}_j^\pr$ above has period coordinates certifying $M \cdot (X,\om) \in \mc{U}^\pr(k_1,k_2)$. The associated homology classes $\wt{a}_1,\dots,\wt{a}_m \in H_1(X;\Z)$ associated to this antler configuration are given by the same expressions, but with $a_j$ in place of $z_j$ and $b_j$ in place of $w_j$.

Another short computation yields the following lemma.

\begin{lem} \label{lem:homologyspan}
Fix $(X,\om) \in \mc{U}^\pr(k_1,k_2)$, and let $\{a_j,b_j\}_{j=1}^g$ be the symplectic basis for $H_1(X;\Z)$ arising from the given antler configuration. Let $b_{12} = b_1 - b_2$, and let
\begin{align*}
\Lam_1 = \Span_\Z \{a_1,\dots,a_{(k_1 + 3)/2}\}, & \quad \Lam_2 = \Span_\Z \{a_1 + b_{12}, a_2 - b_{12}, a_3, \dots, a_{(k_1 + 3)/2}\} \\
\Lam_3 = \Span_\Z \{\wt{a}_1,\dots,\wt{a}_{(k_1 + 3)/2}\}, & \quad \Lam_4 = \Span_\Z \{\wt{a}_1 + b_{12}, \wt{a}_2 - b_{12}, \wt{a}_3, \dots, \wt{a}_{(k_1 + 3)/2}\}
\end{align*}
Then we have
\begin{align*}
\Lam_1 \cap \Lam_2 &= \Span_\Z \{a_1 + a_2, a_3, \dots, a_{(k_1 + 3)/2}\} \\
\Lam_3 \cap \Lam_4 &= \Span_\Z \{\wt{a}_1 + \wt{a}_2, \wt{a}_3, \dots, \wt{a}_{(k_1 + 3)/2}\} \\
(\Lam_1 + \Lam_2) \cap (\Lam_3 + \Lam_4) &= \Span_\Z \{b_1 - b_2, 10a_1 - 9a_2 - 10a_3\}
\end{align*}
\end{lem}

Now fix $(X,\om) \in \mc{U}^\pr(k_1,k_2)$ and let $z_j,w_j,c$ be the period coordinates arising from the given antler configuration. The nearly horizontal closed geodesic $\al_1,\al_2$ in the associated flat tori $T_1,T_2$, and the nearly horizontal slits $\al_3,\dots,\al_{(k_1+3)/2}$ in the antler configuration have homology classes $a_1,\dots,a_{(k_1+3)/2}$. By applying item (4) in Lemma \ref{lem:dehntwistrecurrence}, there is a holomorphic $1$-form $(X_1,\om_1) \in \mc{U}^\pr(k_1,k_2) \cap L_\cC(p_\om)$ for which the analogous homology classes are $a_1 + b_{12}, a_2 - b_{12}, a_3, \dots, a_{(k_1+3)/2}$. Similarly, by applying a Dehn multitwist, and then applying item (4) in Lemma \ref{lem:dehntwistrecurrence}, we obtain $1$-forms in $\GL^+(2,\R) \cdot L_\cC(p_\om)$ for which the analogous homology classes are $\wt{a}_1,\dots,\wt{a}_{(k_1+3)/2}$ and $\wt{a}_1 + b_{12}, \wt{a}_2 - b_{12}, \wt{a}_3, \dots, \wt{a}_g$, respectively. Denote the $\Z$-spans of these $4$ collections of homology classes by $\Lam_1,\Lam_2,\Lam_3,\Lam_4$. We also have analogous subspaces $\Lam_1^\pr,\Lam_2^\pr,\Lam_3^\pr,\Lam_4^\pr$ arising from the nearly vertical geodesics $\bet_1,\bet_2$ and the nearly vertical slits $\al_{(k_1+5)/2},\dots,\al_g$, and their analogous obtained by applying item (2) of Lemma \ref{lem:dehntwistrecurrence} and applying an analogous Dehn multitwist along nearly horizontal closed geodesic. The intersection $\left(\sum_{j=1}^4 \Lam_j\right) \cap \left(\sum_{j=1}^4 \Lam_j^\pr\right)$ contains
\be
\Span_\Z \{a_1,a_2,b_1 - b_2\} \cap \Span_\Z \{b_1,b_2,a_1 - a_2\} = \{a_1 - a_2, b_1 - b_2\}
\ee
and so has rank at least $2$.

We now define a dense open subset $\mc{U}^{\pr\pr}(k_1,k_2) \subset \mc{U}^\pr(k_1,k_2)$ as follows. A holomorphic $1$-form $(X,\om) \in \mc{U}^\pr(k_1,k_2)$ is in $\mc{U}^{\pr\pr}(k_1,k_2)$ if and only if all of the following subgroups of $p_\om(H_1(X;\Z))$ contain a lattice in $\C$:
\be
p_\om(\Lam_1 \cap \Lam_2), \; p_\om(\Lam_3 \cap \Lam_4), \; p_\om((\Lam_1 + \Lam_2) \cap (\Lam_3 + \Lam_4)),
\ee
\be
p_\om(\Lam_1^\pr \cap \Lam_2^\pr), \; p_\om(\Lam_3^\pr \cap \Lam_4^\pr), \; p_\om((\Lam_1^\pr + \Lam_2^\pr) \cap (\Lam_3^\pr + \Lam_4^\pr)),
\ee
\be
p_\om \left( \left(\sum_{j=1}^4 \Lam_j\right) \cap \left(\sum_{j=1}^4 \Lam_j^\pr\right) \right) .
\ee
Note that since each of the given subgroups of $H_1(X;\Z)$ has rank at least $2$, each lattice condition can be satisfied if a particular pair of periods are nonparallel. For instance, if $p_\om(a_1 + a_2)$ and $p_\om(a_3)$ are nonparallel, then $p_\om(\Lam_1 \cap \Lam_2)$ contains a lattice in $\C$.

We are now ready to apply the lemmas from Section \ref{sec:Kapovich} in order to run a similar bootstrap argument for closures of absolute period leaves in $\cC$.

\begin{lem} \label{lem:leafclosurenbhd}
Fix $(X,\om) \in \mc{U}^{\pr\pr}(k_1,k_2)$. Let $a = \Area(X,\om)$, and let $\Lam = \Lam(\om)$ be the closure of $p_\om(H_1(X;\Z))$ in $\C$. Suppose $\Lam(\om)$ is not discrete. Then the closure in $\cC$ of the absolute period leaf $L_\cC(p_\om)$ contains an open neighborhood of $(X,\om)$ in $\cC_a^\Lam$.
\end{lem}

\begin{proof}
We are assuming $p_\om(H_1(X;\Z))$ is not discrete. Let $\{a_j,b_j\}_{j=1}^g$ be the symplectic basis of $H_1(X;\Z)$ arising from given antler configuration on $(X,\om)$. By Lemma \ref{lem:homologyspan} and the discussion preceding Lemma \ref{lem:leafclosurenbhd}, we may assume that $p_\om(\Lam_1)$ is not discrete, where $\Lam_1 = \Span_\Z \{a_1,\dots,a_{(k_1+3)/2}\}$. By definition of $\mc{U}(k_1,k_2)$, the periods $p_\om(a_1),\dots,p_\om(a_{(k_1+3)/2})$ are nonzero and pairwise nonparallel. Since the subgroups $p_\om(\Span_\Z\{a_1,a_2,a_j\})$, $3 \leq j \leq (k_1+3)/2$, all contain the lattice $p_\om(\Span_\Z\{a_1,a_2\})$, and $p_\om(\Lam_1)$ is not discrete, at least one of the $p_\om(\Span_\Z\{a_1,a_2,a_j\})$ must not be discrete. We may assume $p_\om(\Span_\Z\{a_1,a_2,a_3\})$ is not discrete. Write
\be
z = (z_1,z_2,z_3) = (p_\om(a_1),p_\om(a_2),p_\om(a_3)) .
\ee

By Corollary \ref{cor:leafclosureU} and Lemma \ref{lem:UAdim}, there is a closed subgroup $O(z)$ of $\C^{2g}$ of real dimension $d(z)$ (defined in Section \ref{sec:Kapovich}) such that nearby $1$-forms with local period coordinate vectors of the form $(z_1,\dots,w_g) + v$ with $v \in O(z)$ lie in the closure of $L_\cC(p_\om)$. Here, the only components of $O(z)$ that can be nonzero are the components for $w_1,w_2,w_3$.

Suppose $d(z) = 1$. In other words, $A = \Span_\Z \{a_1,a_2,a_3\}$ has $p_\om$-dimension $1$ (defined in Section \ref{sec:Kapovich}). Following the argument in Lemma \ref{lem:UAdim1}, there is a nearby $(X^\pr,\om^\pr)$ in the closure of $L_\cC(p_\om)$ with local period coordinate vector of the form
\be
(z_1,\dots,z_g,w_1^\pr,w_2^\pr,w_3^\pr,w_4,\dots,w_g) \in (z_1,\dots,w_g) + O(z)
\ee
such that $A^\pr = \Span_\Z \{a_1 + (b_1 - b_2), a_2 - (b_1 - b_2), a_3\}$ has $p_{\om^\pr}$-dimension at least $2$, and $p_{\om^\pr}$-type $2$ in the case of equality. By item (4) in Lemma \ref{lem:dehntwistrecurrence}, there is a $1$-form $(X_1,\om_1) \in L_\cC(p_{\om^\pr})$ with local period coordinate vector
\be
(z_1 + (w_1^\pr - w_2^\pr), z_2 - (w_1^\pr - w_2^\pr), z_3,\dots,z_g,w_1^\pr,w_2^\pr,w_3^\pr,w_4,\dots,w_g) .
\ee
Now $A$ has $p_{\om_1}$-dimension at least $2$, and $p_{\om_1}$-type $2$ in the case of equality.

Suppose $A$ has $p_{\om_1}$-type $1$. Note that in this case, $\Lam(\om) = \C$. Following the argument in Lemma \ref{lem:UAtype1}, there is a nearby $(X_1^\pr,\om_1^\pr)$ in the closure of $L_\cC(p_{\om_1})$ such that $A^\pr$ has $p_{\om_1^\pr}$-dimension $5$. Then by item (4) in Lemma \ref{lem:dehntwistrecurrence}, there is $(X_2,\om_2) \in L_\cC(p_{\om_1^\pr})$ such that $A$ has $p_{\om_2}$-dimension $5$. This means we can perturb $p_{\om_2}(b_1),p_{\om_2}(b_2),p_{\om_2}(b_3)$ by any small amount that preserves $\sum_{j=1}^3 \im(\ol{p_{\om_2}(a_j)}p_{\om_2}(b_j))$ while staying in the closure of $L_\cC(p_{\om_2})$. Following the iterative argument in the proof of Theorem \ref{thm:Kapovich}, we conclude that the closure of $L_\cC(p_\om)$ contains an open neighborhood of $(X,\om)$ in $\cC_a$.

Next, suppose $A$ has $p_{\om_1}$-type $2$ and $\Lam(\om) = \C$. Following the argument in Lemma \ref{lem:UAdim2type2}, there is a nearby $(X_1^\pr,\om_1^\pr)$ in the closure of $L_\cC(p_{\om_1})$ such that $A^\pr$ has $p_{\om_1}^\pr$-type $1$. Then by the previous paragraph, we again conclude that the closure of $L_\cC(p_\om)$ contains an open neighborhood of $(X,\om)$ in $\cC_a$.

Lastly, suppose $A$ has $p_{\om_1}$-type $2$ and $\Lam(\om) \cong \R + i\Z$. We may assume $\Lam(\om) = \R + i\Z$. This means we can perturb the real parts of $p_{\om_1}(b_1),p_{\om_1}(b_2),p_{\om_1}(b_3)$ by any small amount that preserves $\sum_{j=1}^3 \im(\ol{p_{\om_1}(a_j)}p_{\om_1}(b_j))$ while staying in the closure of $L_\cC(p_{\om_1})$. Following the iterative argument in the proof of Theorem \ref{thm:Kapovich} again, we conclude that closure of $L_\cC(p_\om)$ contains an open neighborhood of $(X,\om)$ in $\C_a^{\R + i\Z}$.
\end{proof}

\begin{proof} (of Theorem \ref{thm:leafclosures})
Following \cite{Win:ergodic}, Theorem \ref{thm:leafclosures} can be reduced to the case of $2$ distinct zeros by splitting zeros. Thus, it is enough to consider the cases where $\cC$ a either a connected stratum with $2$ distinct zeros, or the nonhyperelliptic component of $\Om\cM_g(g-1,g-1)$ with $g-1 \geq 4$ odd. Letting $k_1 \geq k_2$ be the two zero orders, we have $\mc{U}^{\pr\pr}(k_1,k_2) \subset \cC$.

Fix $(X,\om) \in \mc{U}^{\pr\pr}(k_1,k_2)$ with area $a > 0$. By Lemma \ref{lem:leafclosurenbhd}, if $\Lam(\om)$ is not discrete, then the closure of the absolute period leaf through $(X,\om)$ contains a nonempty open neighborhood in $\cC_a^{\Lam(\om)}$. By Theorems 1.2 and 1.4 in \cite{Win:ergodic}, there are absolute period leaves that are dense in $\cC_a^{\Lam(\om)}$, and thus the absolute period leaf through $(X,\om)$ must be dense in $\cC_a^{\Lam(\om)}$ since it contains such a dense leaf in its closure. It follows that for any $M \in \GL^+(2,\R)$ with determinant $d > 0$, the absolute period leaf through $M \cdot (X,\om)$ is dense in $\cC_{da}^{M \cdot \Lam(\om)}$. Thus, $\mc{U} = \GL^+(2,\R) \cdot \mc{U}^{\pr\pr}(k_1,k_2)$ is a nonempty, open, $\GL^+(2,\R)$-invariant subset of $\cC$ such that for any $(X,\om) \in \mc{U}$ with $\Lam(\om)$, the closure of the absolute period leaf through $(X,\om)$ is dense in $\cC_{\Area(X,\om)}^{\Lam(\om)}$. Moreover, if $\Lam(\om)$ is discrete, then $L_\cC(p_\om)$ is closed and is a connected component $\cC_a^{\Lam(\om)}$. Finally, by \cite{Fil:splitting}, the open set $\mc{U}$ has the form $\mc{U} = \mc{C} \sm \mc{V}$ for some invariant subvariety $\mc{V}$ of positive codimension.
\end{proof}


\bibliographystyle{math}
\bibliography{my.bib}

{\small
\noindent
Email: kgwinsor@gmail.com

\noindent
Department of Mathematics, Stony Brook University, Stony Brook, New York, USA
}

\end{document}